\documentclass[preprint,12pt]{elsarticle}

\usepackage{amssymb}

\usepackage{graphicx}
\usepackage{amsthm}
\usepackage{amsfonts}
\usepackage{amsmath}
\usepackage{amssymb}

\newcommand{\mobs}{\ensuremath{m}}

\newcommand{\yobs}[1]{\ensuremath{y_{#1}}}

\newcommand{\Hil}{\ensuremath{\mathcal{H}}}

\newcommand{\mprob}{\ensuremath{\mathbb{P}}}

\newcommand{\eps}{\varepsilon}

\newcommand{\linop}{\Phi}
\newcommand{\ex}{\operatorname{\mathbb{E}}}

\newcommand{\reals}{\mathbb{R}}
\newcommand{\real}{\mathbb{R}}
\newcommand{\nats}{\mathbb{N}}

\newcommand{\inprod}[2]{\langle#1,#2\rangle}
\newcommand{\defn}{:=}

\newcommand{\order}{\mathcal{O}}
\newcommand{\blems}{\begin{lemma}}
\newcommand{\elems}{\end{lemma}}

\newcommand{\diag}{\operatorname{diag}}
\newcommand{\trace}{\operatorname{tr}}

\newcommand{\fhat}{\widehat{f}}

\newcommand{\Kbb}{\mathbb{K}}

\newcommand{\ball}{B}

\newcommand{\fs}{f^*}

\newcommand{\kereigf}{\psi}
\newcommand{\mnorm}[2]{|\!|\!| #1 |\!|\!|_{#2}}

\newcommand{\iprod}[2]{\langle #1, #2 \rangle}

\newcommand{\Null}{\operatorname{Ker}}

\newcommand{\psdm}[1]{\mathbb{S}^{#1}_+}

\newcommand{\modulo}[2]{(\!(#1)\!)_{#2}}
\newcommand{\sineigf}{\psi^{(s)}}
\newcommand{\coseigf}{\psi^{(c)}}

\newcommand{\qcal}{\mathcal{Q}}

\newcommand{\rlphi}{\underline{R}_\Phi}

\newcommand{\tlphi}{\underline{T}_\Phi}

\newcommand{\Lt}{{L^2}}
\newcommand{\ballh}{B_\Hil}

\newcommand{\xcal}{\mathcal{X}}

\newcommand{\nobs}{n}
\newcommand{\rp}{R_\linop}
\newcommand{\rpp}[1]{R_{#1}}
\newcommand{\tp}{T_\linop}

\newcommand{\psd}[1]{\mathbb{S}_+^{#1}}
\newcommand{\lamax}{\lambda_{\max}}
\newcommand{\lamin}{\lambda_{\min}}

\newcommand{\gcal}{\mathcal{G}}
\newcommand{\ccal}{\mathcal{C}}
\newcommand{\ccalt}{\widetilde{\mathcal{C}}}
\newcommand{\tcal}{\mathcal{T}}

\newcommand{\tcalorth}{\mathcal{T}^\perp}

\newcommand{\crad}{r}

\newcommand{\pr}{\mathbb{P}}
\newcommand{\bigmnorm}[1]{\Big|\!\Big|\!\Big|#1\Big|\!\Big|\!\Big|}

\newcommand{\onebb}{\mathbb{I}}

\newcommand{\qcalt}{\widetilde{\mathcal{Q}}}
\newcommand{\rpl}{\underline{R}_\linop}
\newcommand{\tpl}{\underline{T}_\linop}
\newcommand{\tplo}{\tpl^\circ}

\newcommand{\rpo}{R_\linop^\circ}

\newcommand{\dualf}{\mathcal{L}}
\newcommand{\dualv}{t}

\newcommand{\scal}{\mathcal{S}}
\newcommand{\conv}{\operatorname{conv}}

\newcommand{\BIL}[1]{B_{#1}}
\newcommand{\QUAD}[1]{Q_{#1}}
\newcommand{\Cpsi}{C_\psi}
\newcommand{\Ctpsi}{\widetilde{C}_\psi}

\theoremstyle{plain}
\newtheorem{theorem}{Theorem}

\newtheorem{corollary}{Corollary}
\newtheorem{lemma}{Lemma}

\theoremstyle{definition}

\newtheorem{example}{Example}

\theoremstyle{remark}

\newcommand{\plobs}{\ensuremath{n}}

\newcommand{\Xcal}{\ensuremath{\mathcal{X}}}

\newcommand{\pdim}{\ensuremath{p}}

\newcommand{\Nat}{\ensuremath{\mathbb{N}}}
\newcommand{\hilweight}{\ensuremath{\sigma}}

\newcommand{\alweight}{\ensuremath{\alpha}}
\newcommand{\ALSEQ}{\ensuremath{\{\alweight_k\}_{k=1}^\infty}}
\newcommand{\BESEQ}{\ensuremath{\{\beta_k\}_{k=1}^\infty}}

\newcommand{\MatPsi}[1]{\ensuremath{\Psi_{#1}}}
\newcommand{\PsiMat}[1]{\ensuremath{\MatPsi{#1}}}
\newcommand{\MuMat}[1]{\ensuremath{\Sigma_{#1}}}

\newcommand{\Four}{\ensuremath{\mathbb{T}_{\psi_1^\nobs}}}

\newcommand{\Xsam}{\ensuremath{x_1^\nobs}}

\newcommand{\SamOp}{\ensuremath{\mathbb{S}_{\Xsam}}}

\newcommand{\lammin}{\ensuremath{\lambda_{\operatorname{min}}}}

\newcommand{\Term}{\ensuremath{S}}

\newcommand{\Ker}{\ensuremath{\mathbb{K}}}

\newcommand{\WSamOp}{\ensuremath{\mathbb{W}_{\Xsam,\Wsam}}}
\newcommand{\wei}{\ensuremath{w}}
\newcommand{\Wsam}{\wei_1^{\nobs}}

\newcommand{\elltwo}{{\ell_2}}
\newcommand{\ellone}{{\ell_1}}
\newcommand{\Ltwo}{L^2}
\newcommand{\SEQ}[1]{\{#1_k\}}
\newcommand{\SEQL}[1]{\{#1_k\}_{k=1}^\infty} 
\newcommand{\INTOP}[1]{I_{#1}}
\newcommand{\Cspace}{C}
\newcommand{\epst}{\widetilde{\eps}}
\newcommand{\pdimt}{\widetilde{\pdim}} 
\newcommand{\nobst}{\widetilde{\nobs}}
\newcommand{\periodidx}{I_\nobs}

\newcommand{\Csig}{C_\hilweight}
\newcommand{\Ctsig}{\widetilde{C}_\hilweight}
\newcommand{\muopt}{\mu}
\newcommand{\mobssig}{\mobs_\hilweight}

\newcommand{\dcal}{\mathcal{D}} 
\newcommand{\packnum}{M}

\journal{Journal of Approximation Theory}

\begin{document}

\begin{frontmatter}



\title{Approximation properties of certain operator-induced norms
      on Hilbert spaces}


\author[eecs]{Arash A. Amini}
\ead{amini@eecs.berkeley.edu}
\author[stat,eecs]{Martin J. Wainwright}
\ead{wainwrig@stat.berkeley.edu}

\address[stat]{Department of Statistics and }
\address[eecs]{Department of Electrical Engineering and Computer Sciences
 UC Berkeley,  Berkeley, CA  94720}

\begin{abstract}
We consider a class of operator-induced norms, acting as
finite-dimensional surrogates to the $L^2$ norm, and study
their approximation properties over Hilbert subspaces
of $L^2$. The class includes, as a special case, the usual empirical norm 
encountered, for example, in the context of nonparametric regression
in reproducing kernel Hilbert spaces (RKHS). Our results have
implications to the analysis of $M$-estimators in models based on
finite-dimensional linear approximation of functions, and also to some
related packing problems.
\end{abstract}


\end{frontmatter}


\section{Introduction}
\label{sec:intro}

Given a probability measure $\mprob$ supported on a compact set $\Xcal
\subset \reals^d$, consider the function class
\begin{align}\label{eq:L2:def}
L^2(\mprob) & \defn \big \{ f:\Xcal \rightarrow \real \, \mid\;
\|f\|_{L^2(\mprob)} < \infty \big \},
\end{align}
where $\|f\|_{L^2(\mprob)} \defn \sqrt{\int_{\Xcal} f^2(x) \, d
  \mprob(x)}$ is the usual $L^2$ norm\footnote{We also use
  $L^2(\xcal)$ or simply $L^2$ to refer to the
  space~\eqref{eq:L2:def}, with corresponding conventions for its
  norm. Also, one can take $\xcal$ to be a compact subset of any
  separable metric space and $\mprob$ a (regular) Borel measure.}
defined with respect to the measure $\mprob$.  It is often of interest
to construct approximations to this $L^2$ norm that are
``finite-dimensional'' in nature, and to study the quality of
approximation over the unit ball of some Hilbert space $\Hil$ that is
continuously embedded within $L^2$.  For example, in approximation
theory and mathematical statistics, a collection of $\plobs$ design
points in $\Xcal$ is often used to define a surrogate for the $L^2$
norm.  In other settings, one is given some orthonormal basis of
$L^2(\mprob)$, and defines an approximation based on the sum of
squares of the first $\plobs$ (generalized) Fourier coefficients.  For
problems of this type, it is of interest to gain a precise
understanding of the approximation accuracy in terms of its dimension
$\plobs$ and other problem parameters.

The goal of this paper is to study such questions in reasonable
generality for the case of Hilbert spaces $\Hil$.  We let
$\linop_\nobs : \Hil \to \reals^\nobs$ denote a continuous linear
operator on the Hilbert space, which acts by mapping any $f \in \Hil$
to the $\plobs$-vector $\begin{pmatrix} [\linop_\nobs f]_1 &
[\linop_\nobs f]_2 & \cdots & [\linop_\nobs f]_\nobs
\end{pmatrix}$.  This operator defines the
$\Phi_\nobs$-semi-norm
\begin{align}
\label{eq:phi:norm:intro}
\|f \|_{\linop_\nobs} & \defn \sqrt{\sum_{i=1}^\nobs
[\linop_\nobs f]_i^2}.
\end{align}
In the sequel, with a minor abuse of terminology,\footnote{This can be
  justified by identifying $f$ and $g$ if $\linop f = \linop g$,
  i.e. considering the quotient $\Hil / \ker \linop$.}  we refer to
$\|f\|_{\linop_\nobs}$ as the $\Phi_\nobs$-norm of $f$.  Our goal is
to study how well $\|f\|_{\linop_\nobs}$ approximates $\|f\|_{L^2}$
over the unit ball of $\Hil$ as a function of $\nobs$, and other
problem parameters.  We provide a number of examples of the
\emph{sampling operator} $\linop_\nobs$ in
Section~\ref{sec:linear:ops}.  Since the dependence on the parameter
$\nobs$ should be clear, we frequently omit the subscript to simplify
notation.

In order to measure the quality of approximation over $\Hil$, we
consider the quantity
\begin{align}
\label{eq:rp:def}
  \rp(\eps) & \defn \sup \big \{ \|f\|_\Lt^2 \, \mid \, f \in \ballh,
  \; \|f\|_\linop^2\; \le \eps^2 \big\},
\end{align}
where $B_\Hil \defn \{ f \in \Hil \, \mid \, \|f\|_\Hil \leq 1\}$ is
the unit ball of $\Hil$.  The goal of this paper is to obtain sharp
upper bounds on $\rp$.  As discussed in
Appendix~\ref{app:lower:quant}, a relatively straightforward argument
can be used to translate such upper bounds into lower bounds on the
related quantity
\begin{align}
  \tlphi(\eps) & \defn \inf \;\big\{ \|f\|_\linop^2\; \mid \; f \in
  \ballh, \; \|f\|_\Lt^2 \ge \eps^2 \big\}.
\end{align}
We also note that, for a complete picture of the relationship between
the semi-norm $\|\cdot\|_{\linop}$ and the $L^2$ norm, one can also
consider the related pair
\begin{subequations}
\begin{align}
\tp(\eps) & \defn \sup \big\{ \|f\|_\linop^2 \; \;\mid \, f \in
\ballh, \; \|f\|_\Lt^2 \le \eps^2 \big\} \label{eq:tp:def}, \quad
\mbox{and} \\
  \rlphi(\eps) \;& \defn \, \inf\; \big\{ \|f\|_\Lt^2 \; \mid \; f \in
  \ballh, \; \|f\|_\linop^2\; \ge \eps^2 \big\}.
\end{align}
\end{subequations}
Our methods are also applicable to these quantities, but we limit our
treatment to $(\rp, \tlphi)$ so as to keep the contribution focused.

Certain special cases of linear operators $\linop$, and associated
functionals have been studied in past work.  In the special case $\eps
= 0$, we have
\begin{align*}
\rp(0) & = \sup \big \{ \|f\|_\Lt^2 \, \mid \, f \in \ballh,
  \; \linop(f) = 0 \big \},
\end{align*}
a quantity that corresponds to the squared diameter of $\ballh \cap
\Null(\linop)$, measured in the $\Lt$-norm.  Quantities of this type
are standard in approximation theory (e.g.,~\cite{DeV86,Pin85,Pin86}), for
instance in the context of Kolmogorov and Gelfand widths.  Our primary
interest in this paper is the more general setting with $\eps > 0$,
for which additional factors are involved in controlling $\rp(\eps)$.
In statistics, there is a literature on the case in which $\linop$ is
a sampling operator, which maps each function $f$ to a vector of $n$
samples, and the norm $\|\cdot\|_\linop$ corresponds to the empirical
$L^2$-norm defined by these samples. When these samples are chosen
randomly, then techniques from empirical process theory~\cite{vdG00} can
be used to relate the two terms.  As discussed in the sequel, our
results have consequences for this setting of random sampling.

As an example of a problem in which an upper bound on $\rp$ is useful,
let us consider a general linear inverse problem, in which the goal is
to recover an estimate of the function $\fs$ based on the noisy
observations
\begin{align*}
  \yobs{i} = [\linop \fs]_i + w_i, \quad i =1,\dots, n,
\end{align*}
where $\{w_i\}$ are zero-mean noise variables, and $\fs \in \ballh$ is
unknown.  An estimate $\fhat$ can be obtained by solving a
least-squares problem over the unit ball of the Hilbert space---that
is, to solve the convex program
\begin{align*}
\fhat &  \defn \arg \min_{f\, \in \ballh}\; \sum_{i=1}^n (\yobs{i} -
[\linop f]_i)^2.
\end{align*}
For such estimators, there are fairly standard techniques for deriving
upper bounds on the $\linop$-semi-norm of the deviation $\fhat -\fs$.
Our results in this paper on $\rp$ can then be used to translate this
to a corresponding upper bound on the $L^2$-norm of the deviation
$\fhat - \fs$, which is often a more natural measure of performance.

As an example where the dual quantity $\tpl$ might be helpful,
consider the packing problem for a subset $\dcal \subset \ballh$ of
the Hilbert ball. Let $\packnum(\eps;\dcal,\|\cdot\|_{L^2})$ be the
$\eps$-packing number of $\dcal$ in $\| \cdot\|_{L^2}$, i.e., the
maximal number of function $f_1,\dots,f_\packnum \in \dcal$ such that
$\|f_i - f_j\|_{L^2} \ge \eps$ for all
$i,j=1,\dots,\packnum$. Similarly, let
$\packnum(\eps;\dcal,\|\cdot\|_{\linop})$ be the $\eps$-packing number
of $\dcal$ in $\|\cdot\|_{\linop}$ norm. Now, suppose that for some
fixed $\eps$, $\tpl(\eps) > 0$.  Then, if we have a collection of
functions $\{f_1,\dots, f_\packnum\}$ which is an $\eps$-packing of
$\dcal$ in $\|\cdot\|_{L^2}$ norm, then the same collection will be a
$\sqrt{\tpl(\eps)}$-packing of $\dcal$ in $\|\cdot\|_\linop$. This
implies the following useful relationship between packing numbers
\begin{align*}
  \packnum(\eps\,;\dcal,\|\cdot\|_{L^2}) \le
  \packnum(\sqrt{\tpl(\eps)}\,;\dcal,\|\cdot\|_{\linop}).
\end{align*}

The remainder of this paper is organized as follows.  We begin in
Section~\ref{SecBackground} with background on the Hilbert space
set-up, and provide various examples of the linear operators $\linop$
to which our results apply.  Section~\ref{SecMain} contains the
statement of our main result, and illustration of some its
consequences for different Hilbert spaces and linear operators.
Finally, Section~\ref{SecProofThmUpperOne} is devoted to the proofs of
our results.

\paragraph{Notation:}
For any positive integer $\pdim$, we use $\psd{\pdim}$ to denote the
cone of $\pdim \times \pdim$ positive semidefinite matrices.  For $A,B
\in \psd{p}$, we write $A \succeq B$ or $B \preceq A$ to mean $A - B
\in \psd{p}$. For any square matrix $A$, let $\lamin(A)$ and
$\lamax(A)$ denote its minimal and maximal eigenvalues,
respectively. We will use both $\sqrt{A}$ and $A^{1/2}$ to denote the
symmetric square root of $A \in \psd{p}$. We will use $\SEQ{x} =
\SEQL{x}$ to denote a (countable) sequence of objects (e.g. real-numbers and
functions). Occasionally we might denote an $\nobs$-vector as
$\{x_1,\dots,x_n\}$. The context will determine whether the elements
between braces are ordered. The symbols $\elltwo = \elltwo(\nats)$ are
used to denote the Hilbert sequence space consisting of real-valued
sequences equipped with the inner product
$\iprod{\SEQ{x}}{\SEQ{y}}_{\elltwo} \defn \sum_{k=1}^\infty x_i y_i$. The
corresponding norm is denoted as $\| \cdot \|_{\elltwo}$.


\section{Background}
\label{SecBackground}

We begin with some background on the class of Hilbert spaces of
interest in this paper and then proceed to provide some examples of
the sampling operators of interest.

\subsection{Hilbert spaces}
\label{SecHilBack}
We consider a class of Hilbert function spaces contained within
$L^2(\Xcal)$, and defined as follows. Let $\SEQL{\psi}$ be an
orthonormal sequence (not necessarily a basis) in $\Ltwo(\xcal)$ and
let $\hilweight_1 \geq \hilweight_2 \geq \hilweight_3 \geq \cdots > 0$
be a sequence of positive weights decreasing to zero.  Given these two
ingredients, we can consider the class of functions
\begin{align}
\label{EqnDefnHil}
\Hil & \defn \Big \{ f \in L^2(\mprob) \, \Big|\;\;  f =
\sum_{k=1}^\infty \sqrt{\hilweight_k} \alweight_k \psi_k, \quad
\mbox{for some $\{\alweight_k\}_{k=1}^\infty \in \elltwo(\Nat)$} \Big
\},
\end{align}
where the series in~(\ref{EqnDefnHil}) is assumed to converge in
$\Ltwo$. (The series converges since $\sum_{k=1}^\infty
(\sqrt{\hilweight_k} \alweight_k)^2 \le \hilweight_1 \| \SEQ{\alpha}
\|_{\elltwo} < \infty$.)
We refer to the sequence $\ALSEQ \in \elltwo$ as the
representative of $f$. Note that this representation is unique due to
$\hilweight_k$ being strictly positive for all $k \in \nats$. 

If $f$ and $g$ are two members of $\Hil$, say
with associated representatives $\alweight = \ALSEQ$ and $\beta =
\BESEQ$, then we can define the inner product
\begin{align}
\label{EqnDefnHilIn}
\inprod{f}{g}_{\Hil} & \defn \sum_{k=1}^\infty \alweight_k \beta_k \;
= \; \inprod{\alpha}{\beta}_{\elltwo}.
\end{align}
With this choice of inner product, it can be verified that the space
$\Hil$ is a Hilbert space. (In fact, $\Hil$ inherits all the required
properties directly from $\elltwo$.) For future reference, we note
that for two functions $f, g \in \Hil$ with associated representatives
$\alpha, \beta \in \elltwo$, their $L^2$-based inner product is given
by\footnote{In particular, for $f \in \Hil$, $\|f\|_{L^2} \le
  \sqrt{\hilweight_1} \|f\|_{\Hil}$ which shows that the inclusion
  $\Hil \subset L^2$ is continuous. } $\inprod{f}{g}_{L^2} =
\sum_{k=1}^\infty \hilweight_k \alweight_k \beta_k$.

We note that each $\psi_k$ is in $\Hil$, as it is represented by a
sequence with a single nonzero element, namely, the $k$-th element
which is equal to $\hilweight_k^{-1/2}$. It follows
from~(\ref{EqnDefnHilIn}) that $\iprod{\sqrt{\hilweight_k}
  \psi_k}{\sqrt{\hilweight_j}\psi_j}_\Hil = \delta_{kj}$. That is,
$\{\sqrt{\hilweight_k} \psi_k\}$ is an orthonormal sequence in $\Hil$.
Now, let $f \in \Hil$ be represented by $\alpha \in \elltwo$. We claim
that the series in~(\ref{EqnDefnHil}) also converges in $\Hil$
norm. In particular, $\sum_{k=1}^N \sqrt{\hilweight_k} \alpha_k
\psi_k$ is in $\Hil$, as it is represented by the sequence
$\{\alpha_1,\dots,\alpha_N,0,0,\dots\} \in \elltwo$. It follows
from (\ref{EqnDefnHilIn}) that $\| f -\sum_{k=1}^N \sqrt{\hilweight_k}
\alpha_k \psi_k\|_{\Hil} = \sum_{k=N+1}^\infty \alpha_k^2$ which
converges to $0$ as $N \to \infty$. Thus, $\{ \sqrt{\hilweight_k}
\psi_k\}$ is in fact an orthonormal basis for $\Hil$.
\\

We now turn to a special case of particular importance to us, namely
the reproducing kernel Hilbert space (RKHS) of a continuous
kernel. Consider a symmetric bivariate function $\Ker: \Xcal \times
\Xcal \rightarrow \real$, where $\xcal \subset \reals^d$ is
compact\footnote{Also assume that $\mprob$ assign positive mass to
  every open Borel subset of $\xcal$.}. Furthermore, assume $\Kbb$ to
be positive semidefinite and continuous. Consider the integral
operator $\INTOP{\Kbb}$ mapping a function $f \in L^2$ to the function
$\INTOP{\Kbb}f \defn \int \Kbb(\cdot,y) f(y) d\mprob(y)$. As a
consequence of Mercer's theorem~\cite{RieszSzNagy,Garling2007},
$I_\Kbb$ is a compact operator from $L^2$ to $\Cspace(\xcal)$, the
space of continuous functions on $\xcal$ equipped with the uniform
norm\footnote{In fact, $\INTOP{\Kbb}$ is well defined over $L^1
  \supset L^2$ and the conclusions about $\INTOP{\Kbb}$ hold as a
  operator from $L^1$ to $\Cspace(\xcal)$.}. Let $\SEQ{\hilweight}$ be
the sequence of nonzero eigenvalues of $\INTOP{\Kbb}$, which are
positive, can be ordered in nonincreasing order and converge to
zero. Let $\SEQ{\psi}$ be the corresponding eigenfunctions which are
continuous and can be taken to be orthonormal in $L^2$. With these
ingredients, the space $\Hil$ defined in equation~\eqref{EqnDefnHil}
is the RKHS of the kernel function $\Kbb$. This can be verified as
follows.

As another consequence of the Mercer's theorem, $\Kbb$ has the
decomposition
\begin{align}
  \Ker(x,y) & \defn \sum_{k=1}^\infty \hilweight_k \psi_k(x) \psi_k(y)
\end{align}
where the convergence is absolute and uniform (in $x$ and $y$). In
particular, for any fixed $y \in \xcal$, the sequence $\{
\sqrt{\sigma_k} \psi_k(y)\}$ is in $\elltwo$. (In fact,
$\sum_{k=1}^\infty (\sqrt{\sigma_k} \psi_k(y))^2 = \Kbb(y,y) <
\infty$.) Hence, $\Kbb(\cdot,y)$ is in $\Hil$, as defined
in~(\ref{EqnDefnHil}), with representative $\{\sqrt{\hilweight_k} \:
\psi_k(y) \}$. Furthermore, it can be verified that the
convergence in~(\ref{EqnDefnHil}) can be taken to be also pointwise\footnote{The
  convergence is actually even stronger, namely it is absolute and
  uniform, as can be seen by noting that $\sum_{k=n+1}^m
  |\alpha_k \sqrt{\hilweight_k} \psi_k(y)| \le (\sum_{k=n+1}^m
  \alpha_k^2)^{1/2} (\sum_{k=n+1}^m
  \hilweight_k \psi^2_k(y))^{1/2} \le (\sum_{k=n+1}^m
  \alpha_k^2)^{1/2} \max_{y \in \xcal} k(y,y)$.}. To be
more specific, for any $f \in \Hil$ with representative $\ALSEQ \in \elltwo$, we
have $f(y) = \sum_{k=1}^\infty  \sqrt{\hilweight_k}
\alpha_k \psi_k(y)$, for all $y \in \xcal$. Consequently, by definition of the inner product~\eqref{EqnDefnHilIn}, we
have
\begin{align*}
\inprod{f}{\Ker(\cdot, y)}_{\Hil} & = \sum_{k=1}^\infty
 \alpha_k \sqrt{\hilweight_k} \psi_k(y) \; = \; f(y),
\end{align*}
so that $\Ker(\cdot, y)$ acts as the representer of evaluation.  This
argument shows that for any fixed $y \in \Xcal$, the linear functional
on $\Hil$ given by $f \mapsto f(y)$ is bounded, since we have
\begin{align*}
|f(y)| & = \big| \inprod{f}{\Ker(\cdot, y)}_\Hil \big| \; \leq \;
 \|f\|_\Hil \|\Ker(\cdot, y)\|_\Hil,
\end{align*}
hence $\Hil$ is indeed the RKHS of the kernel $\Kbb$. This fact plays
an important role in the sequel, since some of the linear operators
that we consider involve pointwise evaluation.

A comment regarding the scope: our general results hold for the basic
setting introduced in equation~\eqref{EqnDefnHil}. For those examples
that involve pointwise evaluation, we assume the more refined case of
the RKHS described above.

\subsection{Linear operators, semi-norms and examples}
\label{sec:linear:ops}
Let $\linop : \Hil \to \reals^\nobs$ be a continuous linear operator,
with co-ordinates $[\linop f]_i$ for $i = 1, 2, \ldots, \nobs$.  It
defines the (semi)-inner product
\begin{align}
\inprod{f}{g}_{\linop} & \defn \inprod{\linop
f}{\linop g}_{\real^\nobs},
\end{align}
which induces the semi-norm $\|\cdot\|_\linop$.  By the Riesz
representation theorem, for each $i = 1, \ldots, \nobs$, there is a
function $\varphi_i \in \Hil$ such that $[\linop f]_i =
\iprod{\varphi_i}{f}_\Hil$ for any $f \in \Hil$.  \\
 
\noindent Let us illustrate the preceding definitions with some
examples.
\begin{example}[Generalized Fourier truncation]
\label{ExaGFT}
Recall the orthonormal basis $\{\psi_i\}_{i=1}^\infty$ underlying the
Hilbert space. Consider the linear operator $\Four: \Hil \rightarrow
\real^\nobs$ with coordinates
\begin{align}
\label{eq:freq:trunc:op}
[\Four f]_i & \defn \iprod{\psi_i}{f}_{L^2}, \quad \mbox{for $i = 1,
2, \ldots, \nobs$.}
\end{align}
We refer to this operator as the \emph{(generalized) Fourier
truncation operator,} since it acts by truncating the (generalized)
Fourier representation of $f$ to its first $\nobs$ co-ordinates.  More
precisely, by construction, if $f = \sum_{k=1}^\infty
\sqrt{\hilweight_k} \alweight_k \psi_k$, then 
\begin{equation}\label{eq:freq:tunc:comp}
[\linop f]_i = \sqrt{\hilweight_i} \alweight_i, \qquad \mbox{for $i = 1, 2, \ldots, \nobs$.}
\end{equation}
By definition of the Hilbert inner product, we have $\alweight_i =
\inprod{\psi_i}{f}_\Hil$, so that we can write $[\linop f]_i =
\inprod{\varphi_i}{f}_\Hil$, where $\varphi_i \defn
\sqrt{\hilweight_i} \psi_i$.  \hfill $\diamondsuit$
\end{example}

\begin{example}[Domain sampling]
\label{ExaSample}
A collection $\Xsam \defn \{x_1, \ldots, x_\nobs \}$ of points in the
domain $\xcal$ can be used to define the (scaled) \emph{sampling
operator} $\SamOp: \Hil \rightarrow \real^\nobs$ via
  \begin{align}
\label{eq:time:samp:op}
 \SamOp f & \defn \nobs^{-1/2} \begin{pmatrix} f(x_1) & \ldots &
f(x_\nobs ) \end{pmatrix}, \quad \mbox{for $f \in \Hil$.}
\end{align}
As previously discussed, when $\Hil$ is a reproducing kernel Hilbert
space (with kernel $\Kbb$), the (scaled) evaluation functional $f
\mapsto \nobs^{-1/2}f(x_i)$ is bounded, and its Riesz representation
is given by the function $\varphi_i = \nobs^{-1/2}
\Kbb(\cdot,x_i)$. \hfill $\diamondsuit$
\end{example}

\begin{example}[Weighted domain sampling]
Consider the setting of the previous example. A slight variation on
the sampling operator~(\ref{eq:time:samp:op}) is obtained by adding some
weights to the samples
\begin{align}
  \label{eq:w:time:samp:op}
  \WSamOp f & \defn \nobs^{-1/2} \begin{pmatrix} \wei_1 f(x_1) & \ldots &
\wei_\nobs f(x_\nobs ) \end{pmatrix}, \quad \mbox{for $f \in \Hil$.}
\end{align}
where $\Wsam=(w_1,\dots,w_\nobs)$ is chosen such that $\sum_{k=1}^\nobs \wei_k^2 =
1$. Clearly, $\varphi_i = n^{-1/2} \wei_i\,\Kbb(\cdot,x_i)$.

[As an example of how this might arise, consider approximating $f(t)$
by $\sum_{k=1}^{\nobs}f(x_k)G_\nobs(t,x_k)$ where $\{
G_\nobs(\cdot\,,x_k)\}$ is a collection of functions in
$L^2(\xcal)$ such that $\iprod{G_\nobs(\cdot\,,x_k)}{G_\nobs(\cdot\,,x_j)}_{L^2} =
\nobs^{-1} w_k^2 \,\delta_{kj}$. Proper choices of
$\{G_\nobs(\cdot,x_i)\}$ might produce better approximations to the
$L^2$ norm in the cases where one insists on choosing elements of
$x_1^\nobs$ to be uniformly spaced, while $\mprob$
in~(\ref{eq:L2:def}) is not a uniform distribution. Another slightly
different but closely related case is when one approximates $f^2(t)$
over $\xcal = [0,1]$, by say $\nobs^{-1} \sum_{k=1}^{\nobs-1} f^2(x_k)
W(\nobs(t-x_k))$ for some function $W : [-1,1] \to \reals_+$ and $x_k
= k/\nobs$. Again, non-uniform weights are obtained when $\mprob$ is
nonuniform.]

\hfill
$\diamondsuit$
\end{example}


\section{Main result and some consequences}
\label{SecMain}

We now turn to the statement of our main result, and the development
of some its consequences for various models.

\subsection{General upper bounds on  $\rp(\eps)$}
\label{SecUpperOne}

We now turn to upper bounds on $\rp(\eps)$ which was defined
previously in~(\ref{eq:rp:def}).  Our bounds are stated in terms of a
real-valued function defined as follows: for matrices $D, M \in
\psdm{\pdim}$,
\begin{align}
  \dualf(\dualv, M, D) & \defn \max \biggr \{ \lamax \big( D -\dualv
 \sqrt{D} \, M \sqrt{D} \big), \; 0 \biggr \}, \qquad \mbox{for
 $\dualv \geq 0$.}
\end{align}
Here $\sqrt{D}$ denotes the matrix square root, valid for positive
semidefinite matrices. \\

The upper bounds on $\rp(\eps)$
involve principal submatrices of certain infinite-dimensional matrices---or
equivalently linear operators on $\elltwo(\Nat)$---that we define here.
Let $\MatPsi{}$ be the infinite-dimensional matrix with entries
\begin{align}\label{eq:Psi:def}
[\MatPsi{}]_{jk} & \defn \inprod{\psi_j}{\psi_k}_\linop, \quad
\mbox{for $j, k = 1, 2, \dots $},
\end{align}
and let $\MuMat{} = \diag \{ \hilweight_1, \hilweight_2, \ldots, \}$
be a diagonal operator.  For any $\pdim = 1, 2, \ldots$, we use
$\MatPsi{\pdim}$ and $\MatPsi{\pdimt}$ to denote the principal
submatrices of $\MatPsi{}$ on rows and columns indexed by
$\{1,2,\dots,\pdim\}$ and $\{\pdim+1,\pdim+2,\dots\}$, respectively.
A similar notation will be used to denote submatrices of $\MuMat{}$.
 
\begin{theorem}
\label{ThmUpperOne}
For all $\eps \geq 0$, we have:
\begin{align}
\label{EqnRPStrong}
  \rp(\eps) &\;\le\; \inf_{\pdim \,\in\, \nats} \; \inf_{\dualv \,
 \geq \, 0} \; \Big\{ \dualf(\dualv, \MatPsi{\pdim}, \MuMat{\pdim}) +
 \dualv \, \Big( \eps + 
 \sqrt{\lamax(\MuMat{\pdimt}^{1/2} \MatPsi{\pdimt}^{} \MuMat{\pdimt}^{1/2})}\Big)^2  + \hilweight_{\pdim+1} \Big\}.
\end{align}
Moreover, for any $\pdim \in \Nat$ such that $
\lammin(\PsiMat{\pdim}) > 0$, we have
\begin{align}
\label{EqnRPWeak}
\rp(\eps) & \; \leq \; \Big(1 - \frac{\hilweight_{\pdim+1}}{\hilweight_1} \Big) \;
 \frac{1}{\lammin(\MatPsi{\pdim})} 
 \Big(  \eps +
 \sqrt{\lamax(\MuMat{\pdimt}^{1/2} \MatPsi{\pdimt}^{}
   \MuMat{\pdimt}^{1/2})} \Big)^2 + \hilweight_{\pdim+1}.
\end{align}
\end{theorem}

\paragraph{Remark~(a):} These bounds cannot
be improved in general.  This is most easily seen in the special case
$\eps = 0$.  Setting $\pdim = \nobs$, bound~\eqref{EqnRPWeak}
implies that $\rp(0) \leq \hilweight_{\nobs+1}$ whenever
$\PsiMat{\nobs}$ is strictly positive definite and $\PsiMat{\nobst} = 0$.  This bound is sharp
in a ``minimax sense'', meaning that equality holds if we take the
infimum over all bounded linear operators $\linop: \Hil \rightarrow
\real^\nobs$.  In particular, it is straightforward to show that
\begin{align}\label{eq:rp:minimax}
\inf_{\substack{\linop:\; \Hil \rightarrow \real^\nobs \\ \linop
\;\text{surjective} }} \rp(0) \; &= \inf_{\substack{\linop: \;\Hil
\rightarrow \real^\nobs \\ \linop\; \text{surjective} }}\; \sup_{f \, \in\,
\ballh} \big\{ \|f\|_\Lt^2 \, \mid \; \linop f = 0 \big\} \; = \;
\hilweight_{\nobs+1},
\end{align}
and moreover, this infimum is in fact achieved by some linear
operator.  Such results are known from the general theory of
$n$-widths for Hilbert spaces (e.g., see Chapter IV in
Pinkus~\cite{Pin85} and Chapter~3 of~\cite{Gamkrelidze1990}.)

In the more general setting of $\eps > 0$, there are operators for
which the bound~\eqref{EqnRPWeak} is met with equality. As a simple
illustration, recall the (generalized) Fourier truncation operator $\Four$ from
Example~\ref{ExaGFT}.
First, it can be verified that
$\iprod{\psi_k}{\psi_j}_{\Four} = \delta_{jk}$ for $j,k \le \nobs$ and
$\iprod{\psi_k}{\psi_j}_{\Four} = 0$ otherwise. Taking $p = \nobs$, we
have $\Psi_\nobs = I_\nobs$, that is, the $\nobs$-by-$\nobs$ identity
matrix, and $\Psi_{\nobst} = 0$. Taking $\pdim = \nobs$ in~(\ref{EqnRPWeak}), it follows that for
$\eps^2 \le \hilweight_1$,
\begin{align}
    \rpp{\Four}(\eps)\; &\le\; \Big( 1 -
    \frac{\hilweight_{\nobs+1}}{\hilweight_1}\Big)\eps^2 +
    \hilweight_{\nobs + 1},\label{eq:four:rp:bound} 
\end{align}
As shown in Appendix~\ref{app:Four:details}, the
bound~\eqref{eq:four:rp:bound} in fact holds with equality. In other
words, the bounds of Theorems~\ref{ThmUpperOne} are tight in this
case. Also, note that~(\ref{eq:four:rp:bound}) implies $\rpp{\Four}(0)
\le \hilweight_{\nobs + 1}$ showing that the (generalized) Fourier
truncation operator achieves the minimax bound
of~(\ref{eq:rp:minimax}). Fig~\ref{fig:geom:fourier} provides a
geometric interpretation of these results.

  \begin{figure}[tb]
    \centering
    \includegraphics[scale=0.6]{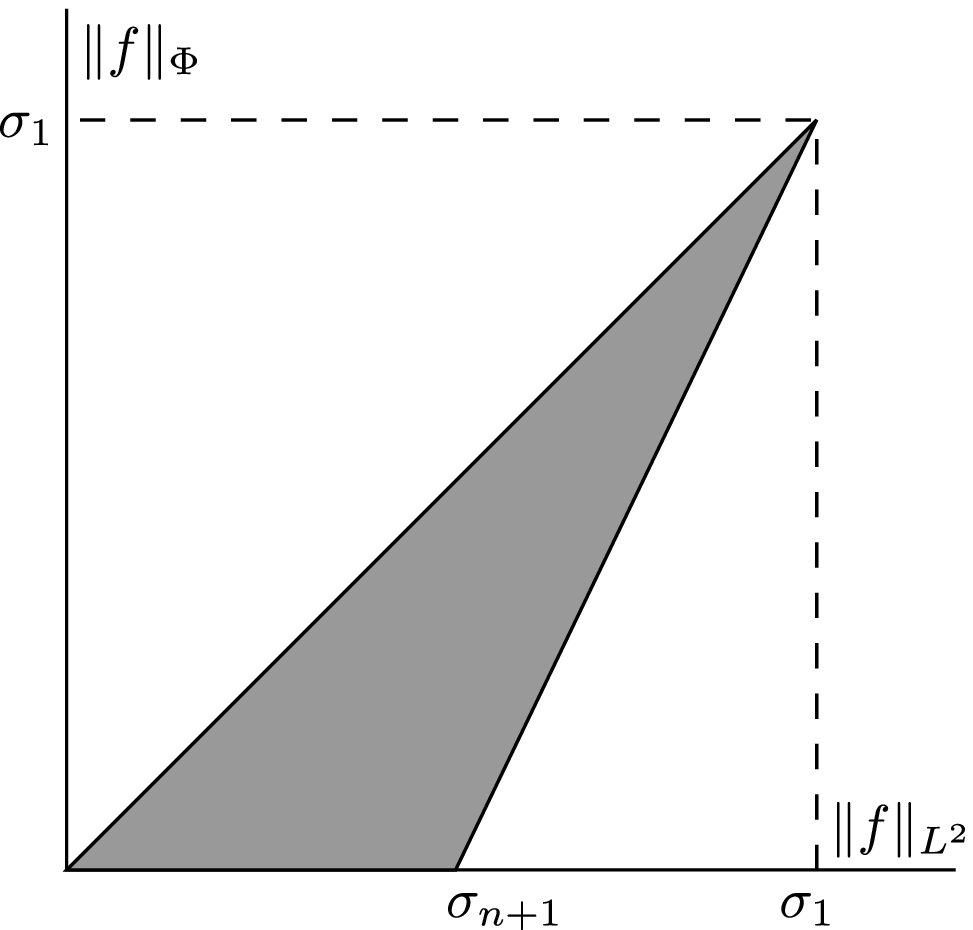}
    \caption{Geometry of Fourier truncation.
       \label{fig:geom:fourier}
    The plot shows the set $\{ (\|f\|_{L^2},\| f\|_\linop):
    \|f\|_{\Hil} \le 1\} \subset \reals^2$ for the case of (generalized) Fourier
    truncation operator $\Four$.}
  \end{figure}

\paragraph{Remark~(b):} In general, it might be difficult to obtain a bound on
$\lamax(\MuMat{\pdimt}^{1/2} \MatPsi{\pdimt}^{} \MuMat{\pdimt}^{1/2})$
as it involves the infinite dimensional matrix
$\MatPsi{\pdimt}^{}$. One may obtain a simple (although not usually
sharp) bound on this quantity by noting that for a positive
semidefinite matrix, the maximal eigenvalue is bounded by the trace,
that is,
   \begin{align}\label{eq:trace:bound}
     \lamax \big( \MuMat{\pdimt}^{1/2} \MatPsi{\pdimt}^{}
     \MuMat{\pdimt}^{1/2} \big)
     \le \trace\big(\MuMat{\pdimt}^{1/2} \MatPsi{\pdimt}^{}
     \MuMat{\pdimt}^{1/2} \big) = \sum_{k \,> \,p} \hilweight_k [\MatPsi{}]_{kk}.
   \end{align}
Another relatively easy-to-handle upper bound is
\begin{align}\label{eq:infnorm:bound}
   \lamax \big( \MuMat{\pdimt}^{1/2} \MatPsi{\pdimt}^{}
     \MuMat{\pdimt}^{1/2} \big)
     \le \mnorm{\MuMat{\pdimt}^{1/2} \MatPsi{\pdimt}^{}
     \MuMat{\pdimt}^{1/2}}{\infty} = \sup_{k \,>\, p} \sum_{r \,> \,p}
     \sqrt{\hilweight_k} \sqrt{\hilweight_r} \big| [\MatPsi{}]_{kr} \big|.
\end{align}
These bounds can be used, in combination with appropriate block
partitioning of $ \MuMat{\pdimt}^{1/2} \MatPsi{\pdimt}^{}
\MuMat{\pdimt}^{1/2}$, to provide sharp bounds on the maximal
eigenvalue. Block partitioning is useful due to the following: for a
positive semidefinite matrix 
$M =
\Big(\begin{smallmatrix}
  A_1 & C \\
  C^T & A_2
\end{smallmatrix}\Big)
$, we have $\lamax(M) \le \lamax(A_1) + \lamax(A_2)$. We leave the the
details on the application of these ideas to examples in
Section~\ref{SecSpecific}.


\subsection{Some illustrative examples}
\label{SecSpecific}

Theorem~\ref{ThmUpperOne} has a number of concrete consequences for
different Hilbert spaces and linear operators, and we illustrate a
few of them in the following subsections.

\subsubsection{Random domain sampling}

We begin by stating a corollary of Theorem~\ref{ThmUpperOne} in
application to random time sampling in a reproducing kernel Hilbert
space (RKHS).  Recall from equation~\eqref{eq:time:samp:op} the time
sampling operator $\SamOp$, and assume that the sample points
$\{x_1,\dots,x_\nobs\}$ are drawn in an i.i.d. manner according to
some distribution $\mprob$ on $\xcal$. Let us further assume that the
eigenfunctions $\psi_k$, $k \ge 1$ are uniformly bounded\footnote{One
  can replace $\sup_{x \in \xcal}$ with essential supremum with
  respect to $\mprob$.}  on $\xcal$, meaning that
  \begin{align}\label{eq:psi:uniform:bounded}
    \sup_{k \ge 1} \,\sup_{x \in \xcal} |\psi_k(x)| \le C_\psi.
  \end{align}
Finally, we assume that $ \|\hilweight\|_1 \defn \sum_{k=1}^\infty
\hilweight_k < \infty$, and that
 \begin{align}
\label{eq:random:samp:assump:sig}
\hilweight_{pk} \le \Csig \,\hilweight_{k}\, \hilweight_{p}, &\quad
\mbox{for some positive constant $\Csig$ and for all large $p$,} \\
     \textstyle \sum_{k > \pdim^\mobs} \hilweight_k \le
     \hilweight_{\pdim}, & \quad
\mbox{for some positive integer $\mobs$ and for all large $p$.} \label{eq:tail:domination}
\end{align}
Let $\mobssig$ be the smallest $\mobs$ for
which~(\ref{eq:tail:domination}) holds. 
These conditions on $\{\hilweight_k\}$ are satisfied, for example, for both a polynomial
decay $\hilweight_k = \order(k^{-\alpha})$ with $\alpha > 1$ and
an exponential decay $\hilweight_k = \order(\rho^{k})$ with $\rho \in (0,1)$.
In particular, for the polynomial decay, using the tail
bound~(\ref{eq:poly:decay:tail:bound}) in
Appendix~\ref{app:sparse:periodic}, we can take $\mobssig = \lceil
\frac{\alpha}{\alpha-1}\rceil$ to
satisfy~(\ref{eq:tail:domination}). For the exponential decay, we can
take $\mobssig = 1$ for $\rho \in (0,\frac{1}{2})$ and $\mobssig = 2$
for $\rho \in (\frac{1}{2},1)$ to
satisfy~(\ref{eq:tail:domination}).

Define the function
   \begin{align}\label{eq:complexity}
     \gcal_\nobs(\eps) := \frac{1}{\sqrt{\nobs}} \sqrt{
         \sum_{j=1}^\infty \min\{\hilweight_j,\eps^2\}},
   \end{align}
as well as the \emph{critical radius}
\begin{align}
\label{eq:crit:rad:def}
 \crad_\nobs := \inf \{ \eps > 0\,:\, \gcal_\nobs(\eps) \le \eps^2\}.
\end{align}

\begin{corollary}\label{cor:random:samp:1}
Suppose that $\crad_\nobs > 0$ and $64 \, C_\psi^2 \,\mobssig\,
\crad_\nobs^2 \log (2\nobs \crad_\nobs^2) \le 1$.  Then for any
$\eps^2 \in [\crad_\nobs^2,\hilweight_1)$, we have
\begin{align}
\pr \Big [ \rpp{\SamOp}(\eps) > (\Ctpsi + \Ctsig)\, \eps^2 \Big] & \le
2 \exp \Big( -\frac{1}{64 \, C_\psi^2\,\crad_\nobs^2}\Big),
\end{align}
where $\Ctpsi \defn 2(1+\Cpsi)^2$ and $\Ctsig \defn 3(1+\Cpsi^{-1})
\Csig \|\hilweight\|_1 + 1$.
\end{corollary}

We provide the proof of this corollary in Appendix~\ref{AppRandSamp}.
As a concrete example consider a polynomial decay $\hilweight_k =
\mathcal{O}(k^{-\alpha})$ for $\alpha > 1$, which satisfies
assumptions on $\{\hilweight_k\}$. Using the tail
bound~(\ref{eq:poly:decay:tail:bound}) in
Appendix~\ref{app:sparse:periodic}, one can verify that $\crad_\nobs^2
= \mathcal{O}(\nobs^{-\alpha/(\alpha+1)})$. Note that, in this case,
\begin{align*}
  \crad_\nobs^2 \log (2 \nobs \crad_n^2)  =
  \order(\nobs^{-\frac{\alpha}{\alpha+1}} \log \nobs^{\frac{1}{\alpha +1}}) =
  \order(\nobs^{-\frac{\alpha}{\alpha+1}} \log \nobs) \to 0, \quad \nobs \to \infty.
\end{align*}
Hence conditions of 
Corollary~\ref{cor:random:samp:1} are met for sufficiently large $\nobs$. It
follows that for some constants $C_1$, $C_2$
and $C_3$, we have
   \begin{align*}
     \rpp{\SamOp}(C_1 \nobs^{-\frac{\alpha}{2(\alpha+1)}}) \le C_2 \, \nobs^{-\frac{\alpha}{\alpha+1}}
   \end{align*}
   with probability $1 - 2\exp(-C_3 \nobs^{\frac{\alpha}{\alpha+1}})$
   for sufficiently large $\nobs$.

\subsubsection{Sobolev kernel}\label{sec:sobolev:ker}

  Consider the kernel $\Kbb(x,y) = \min(x,y)$ defined on $\xcal^2$
  where $\xcal = [0,1]$. The corresponding RKHS is of Sobolev type and
  can be expressed as
  \begin{align*}
    \big\{ f \in L^{2}(\xcal)\,  \mid\, \text{$f$ is
    absolutely continuous, $f(0) = 0$ and $f' \in L^{2}(\xcal)$} \big\}.
  \end{align*}
  Also consider
  a uniform domain sampling operator $\SamOp$, that is, that
  of~(\ref{eq:time:samp:op}) with $x_i = i/\nobs, i \le n$ and let
  $\mprob$ be uniform (i.e., the Lebesgue measure restricted to $[0,1]$).

  This setting has the benefit that many interesting quantities can be
  computed explicitly, while also having some practical appeal. The
  following can be shown about the eigen-decomposition of the integral
  operator $\INTOP{\Kbb}$ introduced in Section~\ref{SecBackground},
  \begin{align*}
    \hilweight_k = \Big[ \frac{(2k-1)\pi}{2}\Big]^{-2}, \quad \psi_k(x) =
    \sqrt{2} \sin \big(\hilweight_k^{-1/2} x\big), \quad k=1,2,\dots.
  \end{align*}
  In particular, the eigenvalues decay as $\hilweight_k =
  \mathcal{O}(k^{-2})$.

  To compute the $\PsiMat{}$, we write
  \begin{align}
    [\PsiMat{}]_{kr} = \iprod{\psi_k}{\psi_r}_\linop = \frac{1}{\nobs}
    \sum_{\ell=1}^\nobs \Big\{ \cos \frac{(k-r) \ell \pi}{\nobs} -
    \cos \frac{(k+r-1)\ell \pi}{\nobs}\Big\}.
  \end{align}
  We note that $\PsiMat{}$ is periodic in $k$ and $r$ with period $2
  \nobs$. It is easily verified that $\nobs^{-1} \sum_{\ell = 1}^\nobs
  \cos(q \ell \pi /\nobs)$ is equal to $-1$ for odd values of $q$ and
  zero for even values, other than $q = 0,\pm 2\nobs, \pm 4\nobs,
  \dots$. It follows that 
  \begin{align}\label{eq:Psi:over:per:sob}
  [\PsiMat{}]_{kr} =
  \begin{cases}
    1 + \frac{1}{n} & \text{if}\; k-r = 0, \\
    -1 -\frac{1}{n} & \text{if}\; k+r = 2n+1 \\
    \frac{1}{n} (-1)^{k-r} & \; \text{otherwise} 
  \end{cases},
\end{align}
for $1 \le k,r \le 2n$. Letting $\onebb_s \in \reals^\nobs$ be the
vector with entries, $(\onebb_s)_j = (-1)^{j+1}, j \le \nobs$, we
observe that $\PsiMat{\nobs}= I_{\nobs} + \frac{1}{\nobs} \onebb_s^{}
\onebb_s^T$. It follows that $\lamin(\PsiMat{\nobs}) =
1$. It remains to bound the terms in~(\ref{EqnRPWeak}) involving the
infinite sub-block $\PsiMat{\nobst}$.

The $\PsiMat{}$ matrix of this example, given by~(\ref{eq:Psi:over:per:sob}),
shares certain properties with the $\PsiMat{}$ obtained in other
situations involving periodic eigenfunctions $\{\psi_k\}$. We abstract away these properties by introducing
a class of periodic $\PsiMat{}$ matrices.
 We call
$\PsiMat{\nobst}$ a \emph{sparse periodic} matrix, if each row (or
column) is periodic and in each period only a vanishing fraction of
elements are large. More precisely, $\PsiMat{\nobst}$ is \emph{sparse
  periodic} if there exist positive integers $\gamma$ and $\eta$, and
positive constants $c_1 $ and $c_2$, all independent of $\nobs$, such
that each row of $\PsiMat{\nobst}$ is periodic with period $\gamma
\nobs.$ and for any row $k$, there exits a subset of elements $S_k =
\{ \ell_1,\dots,\ell_\eta\} \subset \{ 1,\dots,\gamma n\}$ such that
\begin{subequations}
\begin{align}
  \big| [\PsiMat{}]_{k,n+r} \big| &\le c_1, \quad \qquad r \in S_k, \label{eq:sp:per:bnd1}\\
  \big| [\PsiMat{}]_{k,n+r} \big| &\le c_2 \,\nobs^{-1}, \;\quad r \in
  \{1,\dots,\gamma \nobs\} \setminus  S_k, \label{eq:sp:per:bnd2}
\end{align}  
\end{subequations}
The elements of $S_k$ could depend on $k$, but the cardinality of
this set should be the constant $\eta$, independent of $k$
and $\nobs$. Also, note that we are indexing rows and columns of
$\PsiMat{\nobst}$ by $\{ \nobs+1,\nobs+2,\dots\}$; in particular, $k
\ge \nobs + 1$. 
For this class, we have the
following whose proof can be found in Appendix~\ref{app:sparse:periodic}.
\begin{lemma}\label{lem:sp:per}
  Assume $\PsiMat{\nobst}$ to be sparse periodic as defined above and
  $\hilweight_k = \mathcal{O}( k^{-\alpha})$, $\alpha \ge 2$. Then,
  \begin{itemize}
  \item[(a)] for $\alpha > 2$, $\lamax \big( \MuMat{\nobst}^{1/2}
    \MatPsi{\nobst}^{} \MuMat{\nobst}^{1/2} \big) =  \mathcal{O}(
    \nobs^{-\alpha})$, $\nobs \to \infty$,
  \item[(b)] for $\alpha = 2$, $\lamax \big( \MuMat{\nobst}^{1/2}
    \MatPsi{\nobst}^{} \MuMat{\nobst}^{1/2} \big) =  \mathcal{O}(
    \nobs^{-2}\log \nobs)$, $\nobs \to \infty$.
  \end{itemize}
\end{lemma}
In particular~(\ref{eq:Psi:over:per:sob}) implies that $\PsiMat{\nobst}$
is sparse periodic  with
parameters $\gamma = 2$, $\eta = 2$, $c_1 = 2$ and $c_2 = 1$. Hence,
part~(b) of Lemma~\ref{lem:sp:per} applies. Now, we can
use~(\ref{EqnRPWeak}) with $\pdim = \nobs$  to obtain
\begin{align}
  \rpp{\SamOp}(\eps) \le 2 \eps^2 + \mathcal{O} \big( \nobs^{-2}
  \log \nobs\big)
\end{align}
where we have also used $(a + b)^2 \le 2 a^2 + 2b^2$.
  
\begin{figure}[tb]
  \centering
  \begin{minipage}[h]{0.4\linewidth}
      \centering
      \includegraphics[scale=0.5]{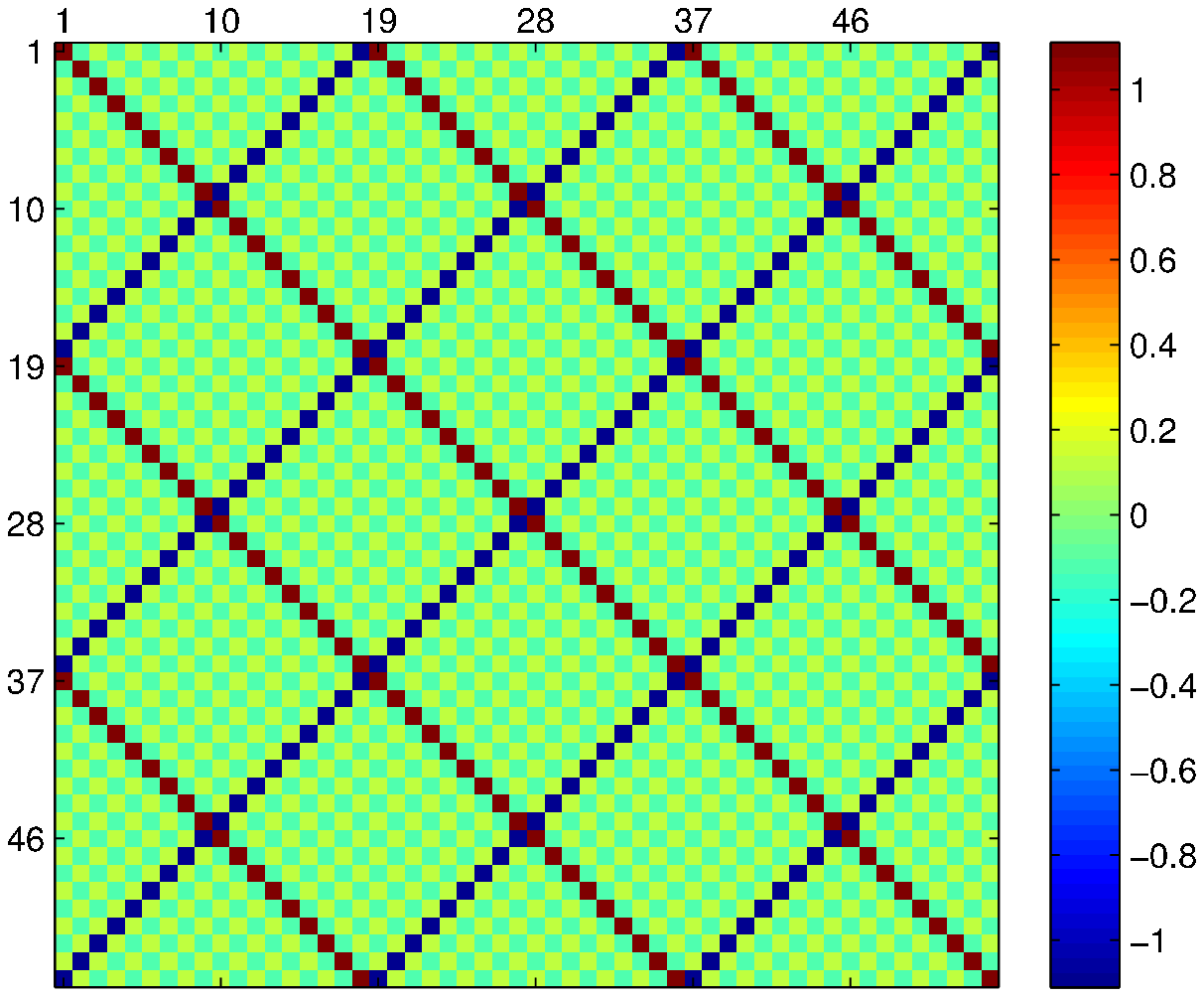} \\
      (a)
  \end{minipage} 
  \hspace{4ex}
  \begin{minipage}[h]{0.4\linewidth}
      \centering
      \includegraphics[scale=0.5]{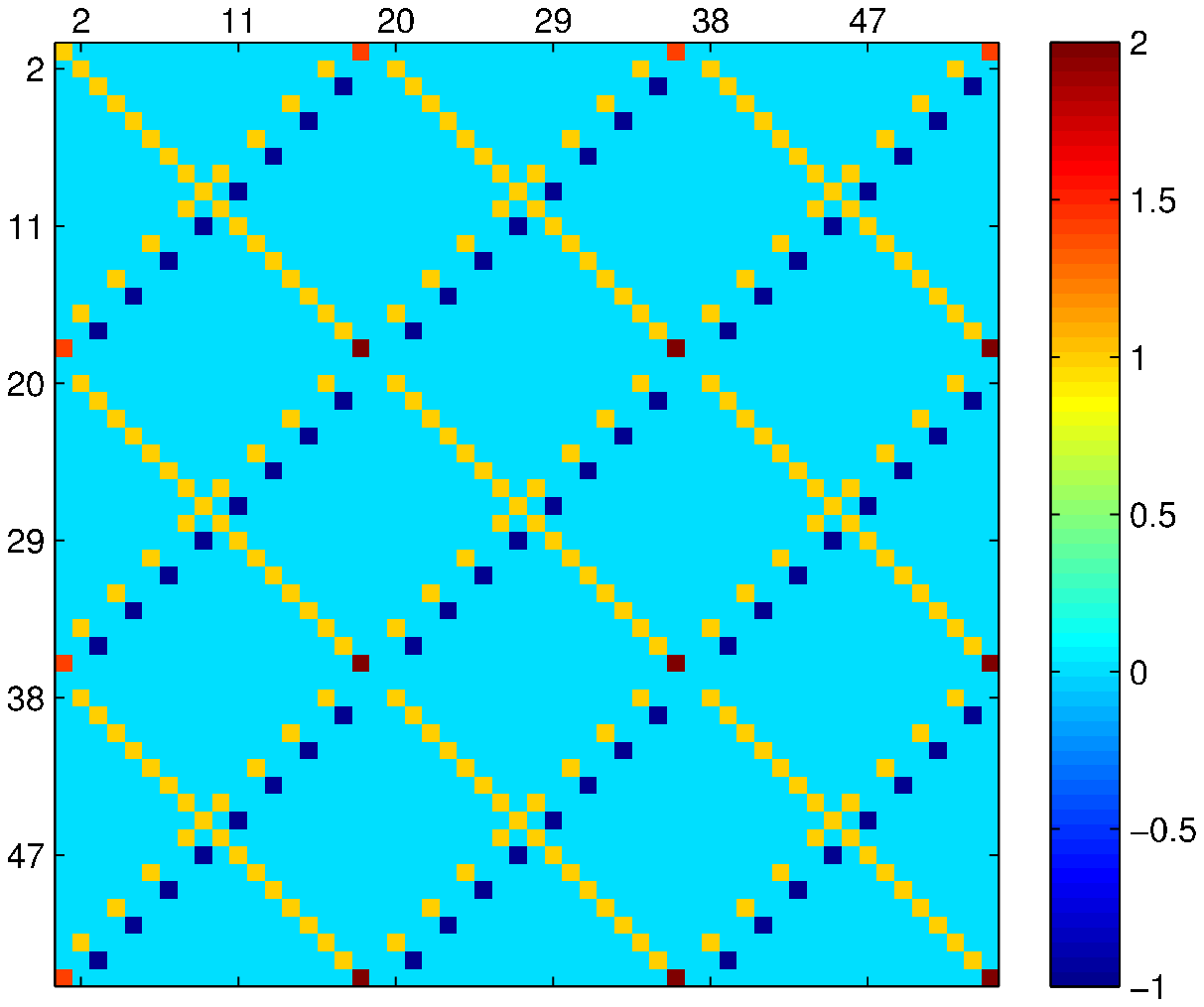} \\
      (b)
  \end{minipage}
  \caption{Sparse periodic $\PsiMat{}$
    matrices. Display~(a) is a plot of the $N$-by-$N$ leading
    principal submatrix of $\PsiMat{}$ for the Sobolev
    kernel $(s,t) \mapsto \min\{s,t\}$. Here $\nobs = 9$ and $N =
    6\nobs$; the period is $2\nobs = 18$. Display~(b) is a the same
    plot for a Fourier-type kernel. The plots exhibit sparse periodic
    patterns as defined in Section~\ref{sec:sobolev:ker}. \label{fig:sp:per}}
\end{figure}

\subsubsection{Fourier-type kernels}
  In this example, we consider an RKHS of functions on $\xcal = [0,1]
  \subset \reals$, generated by a \emph{Fourier-type} kernel defined
  as $\Kbb(x,y) := \kappa(x-y)$, $x,y \in [0,1]$, where
  \begin{align}\label{eq:trig:series}
    \kappa(x) = \zeta_0 + \sum_{k=1}^\infty 2 \zeta_k \cos(2 \pi k x),
    \quad x \in [-1,1].
  \end{align}
  We assume that $(\zeta_k)$ is a $\reals_+$-valued nonincreasing
  sequence in $\ellone$, i.e. $\sum_k \zeta_k < \infty$. Thus, the
  trigonometric series in~(\ref{eq:trig:series}) is absolutely (and
  uniformly) convergent. As for the operator $\linop$, we consider the
  uniform time sampling operator $\SamOp$, as in the previous
  example. That is, the operator defined in~(\ref{eq:time:samp:op})
  with $x_i = i/\nobs, i \le n$. We take $\mprob$ to be uniform.

  This setting again has the benefit of being simple enough to allow
  for explicit computations while also practically important. One can
  argue that the eigen-decomposition of the kernel integral operator
  is given by
  \begin{align}
    \kereigf_1 &= \coseigf_0, \quad \kereigf_{2k} = \coseigf_k,
  \quad 
  \kereigf_{2k+1} = \sineigf_k, \quad k \ge 1\\
  \hilweight_1 &= \zeta_0, \quad\; \;\; \hilweight_{2k} = \zeta_{k}, 
 \;\;\; \quad \hilweight_{2k+1} = \zeta_{k}, \quad \;\;\;k \ge 1
  \end{align}
  where $\coseigf_0(x) := 1$, $\coseigf_k(x) := \sqrt{2} \cos(2\pi k
  x)$ and $\sineigf_k(t) := \sqrt{2} \sin(2\pi k x)$ for $k \ge 1$.

  For any integer $k$, let $\modulo{k}{\nobs}$ denote $k$ modulo
  $\nobs$. Also, let $k \mapsto \delta_k$ be the function defined over
  integers which is $1$ at $k = 0$ and zero elsewhere. Let $\iota
  \defn \sqrt{-1}$. Using the identity $\nobs^{-1} \sum_{\ell=1}^\nobs
  \exp(\iota 2 \pi k \ell / \nobs) = \delta_{\modulo{k}{\nobs}}$, one
  obtains the following, 
\begin{subequations}
  \begin{align}
    \iprod{\coseigf_k}{\coseigf_j}_\linop &= \big[
      \delta_{\modulo{k-j}{\nobs}} + \delta_{\modulo{k+j}{\nobs}}\big]
    \Big(\frac{1}{\sqrt{2}}\Big)^{\delta_k + \delta_j},
    \\ \iprod{\sineigf_k}{\sineigf_j}_\linop &=
    \delta_{\modulo{k-j}{\nobs}} - \delta_{\modulo{k+j}{\nobs}},
    \\ \iprod{\coseigf_k}{\sineigf_j}_\linop &= 0, \qquad \mbox{ valid
      for all $j,k \ge 0$.}
  \end{align}
\end{subequations}
 It follows that $\PsiMat{\nobs} = I_\nobs$ if
  $\nobs$ is odd and $\Psi_\nobs = \diag\{1,1,\dots,1,2\}$ if $\nobs$
  is even. In particular, $\lamin(\PsiMat{\nobs}) = 1$ for all $\nobs
  \ge 1$. It is also clear that the principal submatrix of $\PsiMat{}$
  on indices $\{2,3,\dots\}$ has periodic rows and columns with period
  $2 \nobs$. If follows that $\PsiMat{\nobs}$ is sparse periodic as
  defined in Section~\ref{sec:sobolev:ker} with parameters $\gamma = 2$,
  $\eta = 2$, $c_1 = 2$ and $c_2 = 0$.

  Suppose for example that the eigenvalues decay polynomially, say as
  $\zeta_k = \mathcal{O}(k^{-\alpha})$ for $\alpha >2$. Then,
  applying~(\ref{EqnRPWeak}) with $\pdim = \nobs$, in
  combination with Lemma~\ref{lem:sp:per}~part~(a), we get
  \begin{align}
    \rpp{\SamOp}(\eps) \le 2 \eps^2 + \mathcal{O}(\nobs^{-\alpha}).
  \end{align}
  As another example, consider the exponential decay $\zeta_k =
  \rho^k$, $k \ge 1$ for some $\rho \in (0,1)$, which corresponds to
  the Poisson kernel. In this case, the tail sum of $\{\hilweight_k\}$
  decays as the sequence itself, namely, $\sum_{k > \nobs} \hilweight_k \le 2
  \sum_{k > \nobs} \rho^k = \frac{2\rho}{1-\rho} \rho^k$. Hence, we
  can simply use the trace bound~(\ref{eq:trace:bound}) together
  with~(\ref{EqnRPWeak}) to obtain
  \begin{align}
    \rpp{\SamOp}(\eps) \le 2 \eps^2 + \mathcal{O}(\rho^{\nobs}).
  \end{align}

\section{Proof of Theorem~\ref{ThmUpperOne} }
\label{SecProofThmUpperOne}
We now turn to the proof of our main theorem.  Recall from
Section~\ref{SecHilBack} the correspondence between any $f \in \Hil$
and a sequence $\alpha \in \elltwo$; also, recall the diagonal
operator $\MuMat{}:\elltwo \rightarrow \elltwo$ defined by the matrix
$\diag \{ \hilweight_1, \hilweight_2, \ldots \}$.  Using the
definition of~(\ref{eq:Psi:def}) of the $\PsiMat{}$ matrix, we have
\begin{align*}
   \|f \|_\linop^2 & = \iprod{\alpha}{ \MuMat{}^{1/2} \MatPsi{}
   \MuMat{}^{1/2}\alpha}_{\elltwo},
\end{align*}
By definition~\eqref{EqnDefnHil} of the
Hilbert space $\Hil$, we have $\|f\|_\Hil^2 = \sum_{k=1}^\infty
\alpha_k^2$ and $\|f\|_{L^2}^2 = \sum_k \hilweight_k \alpha_k^2$.
Letting $\ball_{\elltwo} = \big \{ \alpha \in \elltwo \, \mid \,
\|\alpha\|_{\elltwo} \le 1 \big\}$ be the unit ball in
$\elltwo$, we conclude that $\rp$ can be written as
\begin{align}
\label{eq:rp:alpha:1}
\rp(\eps) & = \sup_{\alpha \,\in\, \ball_{\elltwo}} \big\{ Q_2(\alpha) \,
\mid \, Q_\linop(\alpha) \leq \eps^2 \big\},
\end{align}
where we have defined the quadratic functionals
\begin{align}\label{eq:Q2:QPhi:defs}
  Q_2(\alpha) \defn \iprod{\alpha}{\MuMat{} \alpha}_{\elltwo}, \quad
  \mbox{and} \quad Q_\linop(\alpha) \defn
  \iprod{\alpha}{\MuMat{}^{1/2} \MatPsi{} \MuMat{}^{1/2}
  \alpha}_{\elltwo}.
\end{align}
Also let us define the symmetric bilinear form
\begin{align}
  \BIL{\linop}(\alpha,\beta) \defn \iprod{\alpha}{\MuMat{}^{1/2}
    \MatPsi{} \MuMat{}^{1/2} \beta}_{\elltwo}, \quad \alpha, \beta \in \ell^2,
\end{align}
whose diagonal is $\BIL{\linop}(\alpha,\alpha) = \QUAD{\linop}(\alpha)$.

We now upper bound $\rp(\eps)$ using a truncation argument.  
Define the set
\begin{align}
  \ccal & \defn \{ \alpha \in \ball_{\elltwo}\,\mid\, Q_\linop(\alpha) \le
  \eps^2\},
\end{align}
corresponding to the feasible set for the optimization
problem~\eqref{eq:rp:alpha:1}.  For each integer $\pdim = 1, 2,
\ldots$, consider the following truncated sequence spaces
\begin{align*}
  \tcal_\pdim & \defn \big \{ \alpha \in \elltwo \, \mid \, \alpha_i =
  0, \quad \mbox{for all $i > p$} \big \}, \quad \mbox{and} \\
\tcalorth_\pdim & \defn \big\{ \alpha \in \elltwo \, \mid \, \alpha_i =
  0, \quad \mbox{for all $i = 1,2, \ldots \pdim$} \big \}.
\end{align*}
Note that $\elltwo$ is the direct sum of $\tcal_\pdim$ and
$\tcalorth_\pdim$. Consequently, any fixed $\alpha \in \ccal$ can be
decomposed as $\alpha = \xi + \gamma$ for some (unique) $ \xi \in
\tcal_p$ and $\gamma \in \tcalorth_p$.  Since $\MuMat{}$ is a diagonal
operator, we have
\begin{align*}
 Q_2(\alpha) & = Q_2(\xi)+ Q_2(\gamma).
\end{align*}
Moreover, since any $\alpha \in \ccal$ is feasible for the
optimization problem~\eqref{eq:rp:alpha:1}, we have
\begin{align}
\label{eq:quad:feasib}
\QUAD{\linop}(\alpha) \, = \, Q_\linop(\xi) + 
2 \BIL{\linop}(\xi,\gamma) + \QUAD{\linop}(\gamma) \; \le \; \eps^2.
\end{align}
Note that since $\gamma \in \tcal_\pdim^\perp$, it can be written as
$\gamma = (0_\pdim, c)$, where $0_\pdim$ is a vector of $\pdim$
zeroes, and $c = (c_1, c_2, \ldots) \in \elltwo$. Similarly, we can
write $\xi =(x,0)$ where $x \in \reals^\pdim$. Then, each of the terms
$\QUAD{\linop}(\xi)$, $\BIL{\linop}(\xi,\gamma)$,
  $\QUAD{\linop}(\gamma)$ can be expressed in terms of block
  partitions of $\MuMat{}^{1/2} \MatPsi{} \MuMat{}^{1/2}$.  For example,
\begin{align}
  \QUAD{\linop}(\xi) = \iprod{x}{A x}_{\real^p}, \quad \QUAD{\linop}(\gamma) = 
  \iprod{y}{ D y}_{\elltwo},
\end{align}
where $A \defn \MuMat{\pdim}^{1/2}\MatPsi{\pdim}
\MuMat{\pdim}^{1/2}$ and $D \defn \MuMat{\pdimt}^{1/2}\MatPsi{\pdimt}
\MuMat{\pdimt}^{1/2}$, in correspondence with the block partitioning
notation of Appendix~\ref{app:quad:ineq}. We now apply inequality~(\ref{eq:quad:ineq}) derived in
Appendix~\ref{app:quad:ineq}. Fix some $\rho^2 \in (0,1)$ and take
\begin{align}
  \kappa^2 \defn \rho^2 \lamax(\MuMat{\pdimt}^{1/2} \MatPsi{\pdimt}
  \MuMat{\pdimt}^{1/2}),
\end{align}
so that condition~(\ref{eq:quad:ineq:cond:2}) is
satisfied. Then,~(\ref{eq:quad:ineq}) implies
\begin{align}\label{eq:quad:ineq:applied}
  \, Q_\linop(\xi) + 
2 \BIL{\linop}(\xi,\gamma) + \QUAD{\linop}(\gamma) \ge \rho^2 \QUAD{\linop}(\xi) -
  \frac{\kappa^2}{1-\rho^2} \|\gamma\|_2^2.
\end{align}
Combining~(\ref{eq:quad:feasib}) and~(\ref{eq:quad:ineq:applied}), we
obtain
\begin{align}
  \QUAD{\linop}(\xi) \le \frac{\eps^2}{\rho^2} +
  \frac{\lamax(\MuMat{\pdimt}^{1/2} \MatPsi{\pdimt}
    \MuMat{\pdimt}^{1/2})}{1-\rho^2} \|\gamma\|_2^2.
\end{align}
We further note that $\|\gamma\|_2^2 \le \|\gamma\|_2^2 + \|\xi\|_2^2
= \|\alpha\|_2^2 \le 1$. It follows that
\begin{align}
  \QUAD{ \linop }(\xi)\; \le\; \epst^2, \quad \text{where} \quad
  \epst^2 \defn \frac{\eps^2}{\rho^2} +
  \frac{\lamax(\MuMat{\pdimt}^{1/2} \MatPsi{\pdimt}
    \MuMat{\pdimt}^{1/2})}{1-\rho^2}.
\end{align}

Let us define 
\begin{align}
  \ccalt & \defn \{ \xi \in \ball_{\elltwo} \cap \tcal_\pdim \,\mid\,
  Q_\linop(\xi) \le \epst^2\}.
\end{align}
Then, our arguments so far show that for $\alpha \in \ccal$,
\begin{align}\label{eq:first:bound:on:QUAD2}
  \QUAD{2}(\alpha) = \QUAD{2}(\xi) + \QUAD{2}(\gamma) \; \le \;
  \underbrace{ \sup_{\xi \, \in\, \ccalt}
    \;\QUAD{2}(\xi) }_{\Term_\pdim} \;+ \underbrace{\sup_{\gamma \,
      \in\, \ball_{\elltwo} \cap \tcalorth_\pdim}
    \;\QUAD{2}(\gamma)}_{\Term_\pdim^\perp}.
\end{align}
Taking the supremum over $\alpha \in \ccal$ yields the upper bound
\begin{align*}
\rp(\eps) & \leq \Term_\pdim + \Term_\pdim^\perp.
\end{align*}

It remains to bound each of the two terms on the right-hand side.
Beginning with the term $\Term_\pdim^\perp$ and recalling the
decomposition $\gamma = (0_\pdim, c)$, we have $Q_2(\gamma) =
\sum_{k=1}^\infty \hilweight_{k+\pdim} c_k^2$, from which it follows
that
\begin{align*}
\Term_\pdim^\perp & =\sup \Big\{ \sum_{k=1}^\infty
\hilweight_{k+\pdim}\, c_k^2 \; \mid \; \sum_{k=1}^\infty c_k^2 \le
1\Big\} \; = \; \hilweight_{\pdim+1},
\end{align*}
since $\{\hilweight_{k}\}_{k=1}^\infty$ is a nonincreasing sequence by
assumption.

We now control the term $\Term_\pdim$.  Recalling the decomposition
$\xi = (x, 0)$ where $x \in
\reals^\pdim$, we have
\begin{align*}
  \Term_\pdim = \sup_{\xi\, \in\, \ccalt} \, Q_2(\xi) & = \sup
  \big\{ \iprod{x}{\MuMat{\pdim} \,x} \,:\, \iprod{x}{x} \le 1,\;
  \iprod{x}{\MuMat{\pdim}^{1/2} \PsiMat{\pdim} \MuMat{\pdim}^{1/2} \,x}
  \le \epst^2\big\} \\ 
& = \sup_{\iprod{x}{x} \,\le\, 1} \, \inf_{t \, \geq \, 0} \big\{
\iprod{x}{\MuMat{\pdim} x} + t \big (\epst^2 - \iprod{x}{\MuMat{\pdim}^{1/2}
\PsiMat{\pdim} \MuMat{\pdim}^{1/2}\, x} \big) \big\} \\
& \stackrel{(a)}{\leq} \inf_{t \, \geq \, 0} \big\{ \sup_{\iprod{x}{x}
\,\le\, 1} \iprod{x}{\MuMat{\pdim}^{1/2}(I_p - t \PsiMat{\pdim})
\MuMat{\pdim}^{1/2} \,x} +t \,\epst^2 \big\}
\end{align*}
where inequality (a) follows by Lagrange (weak) duality. It is not
hard to see that for any symmetric matrix $M$, one has
\begin{align*}
 \sup \big\{ \iprod{x}{M x} \, :\, \iprod{x}{x} \le 1\big\} & = \max
  \big \{ 0, \lamax(M) \big \}.
\end{align*}
Putting the pieces together and optimizing over $\rho^2$, noting that
\begin{align*}
  \inf_{r \in (0,1)} \Big\{ \frac{a}{r} + \frac{b}{1-r}
\Big\} = (\sqrt{a} + \sqrt{b})^2
\end{align*}
for any $a, b > 0$,
completes the proof of the
bound~\eqref{EqnRPStrong}. \\

We now prove bound~\eqref{EqnRPWeak}, using the same decomposition
and notation established above, but writing an upper bound on
$\QUAD{2}(\alpha)$slightly different
form~(\ref{eq:first:bound:on:QUAD2}). In particular, the argument
leading to~(\ref{eq:first:bound:on:QUAD2}), also shows that
\begin{align}
  \rp(\eps) \le \sup_{\xi \, \in \, \tcal_\pdim, \; \gamma \, \in \,
    \tcalorth_\pdim} \big\{ \QUAD{2}(\xi) + \QUAD{2}(\gamma)\; \mid \;
  \xi + \gamma \in \ball_{\elltwo}, \; \QUAD{\linop}(\xi) \le \epst^2
  \big\}.
\end{align}
Recalling the expression~(\ref{eq:Q2:QPhi:defs}) for $\QUAD{\linop}(\xi)$
and noting that $\PsiMat{\pdim} \succeq \lammin(\PsiMat{\pdim})
I_\pdim$ implies $A = \MuMat{\pdim}^{1/2} \PsiMat{\pdim}
\MuMat{\pdim}^{1/2} \succeq \lammin(\PsiMat{\pdim}) \MuMat{\pdim}$, we
have
\begin{align}
  Q_\linop(\xi) \;\ge \;\lammin(\PsiMat{\pdim}) \,Q_2(\xi).
\end{align}
Now, since we are assuming $\lammin(\PsiMat{\pdim}) > 0$, we have
\begin{align}
  \rp(\eps) \le \sup_{ \xi
      \, \in \, \tcal_\pdim, \; \gamma \, \in \, \tcalorth_\pdim } 
  \Big\{ \QUAD{2}(\xi) + \QUAD{2}(\gamma)\;\; \Big| \;\;
  \xi + \gamma \in \ball_{\elltwo}, \; \QUAD{2}(\xi) \le
  \frac{\epst^2}{\lammin(\PsiMat{\pdim})} \Big\}.
\end{align}
The RHS of the above is an instance of the Fourier truncation problem
with $\eps^2$ replaced with $\epst^2 / \lammin(\PsiMat{\pdim})$. That
problem is workout in detail in Appendix~\ref{app:Four:details}. In
particular, applying equation~(\ref{eq:Four:exact:appendix}) in
Appendix~\ref{app:Four:details} with $\eps^2$ changed to $\epst^2 /
\lammin(\PsiMat{\pdim})$ completes the proof of~\eqref{EqnRPWeak}.
Figure~\ref{fig:proof:geom} provides a graphical representation of the
geometry of the proof.

\begin{figure}[tb]
  \centering
  \begin{minipage}[h]{0.4\linewidth}
      \centering
      \includegraphics[scale=0.5]{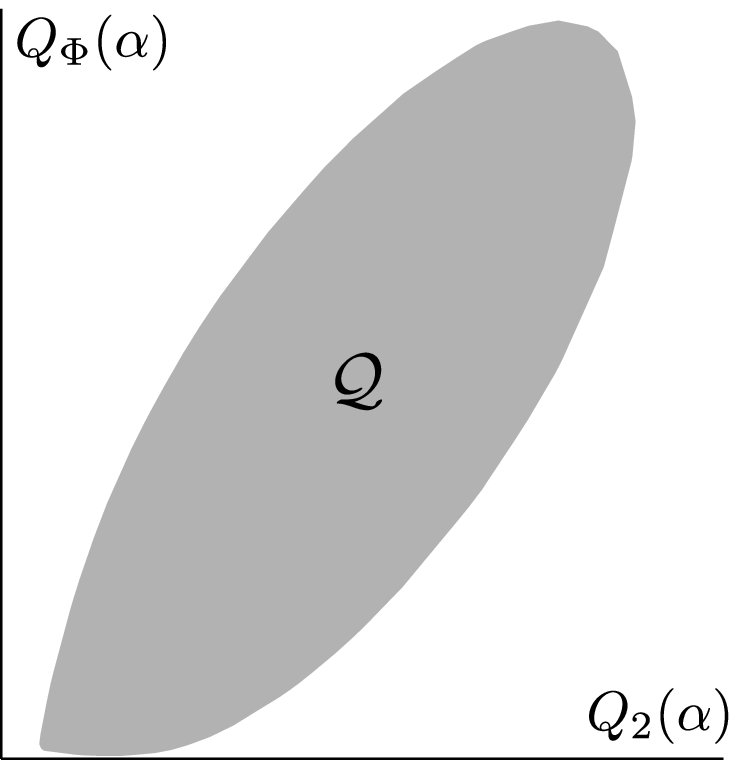} \\
      (a)
  \end{minipage} 
  \begin{minipage}[h]{0.4\linewidth}
      \centering
      \includegraphics[scale=0.5]{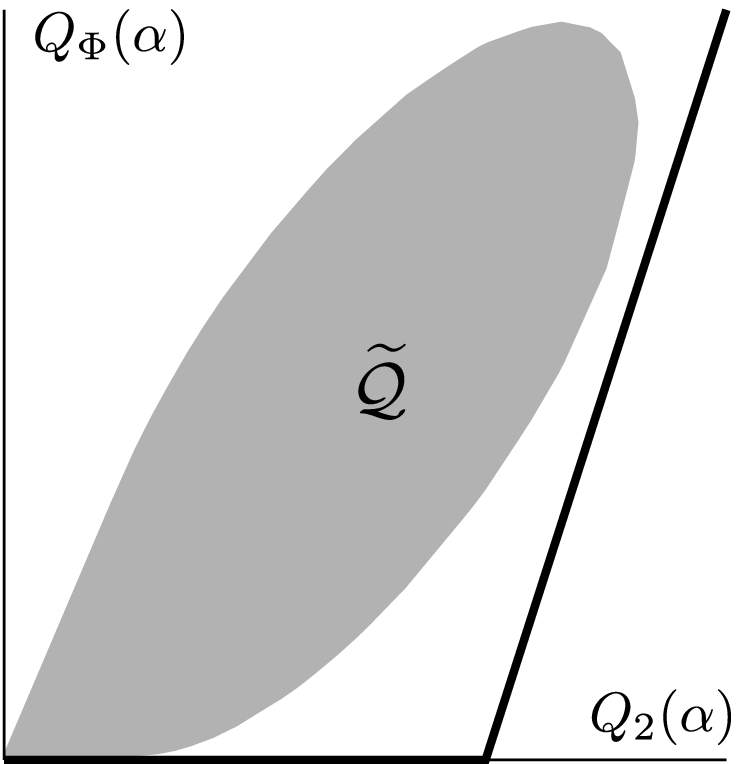} \\
      (b)
  \end{minipage}
  \caption{Geometry of the proof of~(\ref{EqnRPWeak}). Display~(a) is
    a plot of the set $\qcal := \{ (Q_2(\alpha),
    Q_\linop(\alpha))\,:\, \|\alpha\|_{\elltwo} = 1\} \subset
    \reals^2$. This is a convex set as a consequence of
    Hausdorff-Toeplitz theorem on convexity of the numerical range and
    preservation of convexity under projections. Display~(b) shows the
    set $\qcalt:=\text{conv}(0,\qcal)$, i.e., the convex hull of
    $\{0\} \cup \qcal$. Observe that
    $\rp(\eps) = \sup \{ x\, : (x,y) \in \qcalt,\; y \le
    \eps^2\}$. For any fixed $r \in (0,1)$, the bound
    of~(\ref{EqnRPWeak}) is a piecewise linear approximation to one
    side of $\qcalt$ as shown in Display~(b).\label{fig:proof:geom}}
\end{figure}


\section{Conclusion}

We considered the problem of bounding (squared) $L^2$ norm of 
functions in a Hilbert unit ball, based on restrictions on an operator-induced norm acting as a surrogate for the $L^2$ norm. In particular, given that
$f \in \ball_\Hil$ and $\|f\|_\linop^2 \le \eps^2$, our results enable
us to obtain, by estimating norms of certain finite and infinite
dimensional matrices, inequalities of the form
\begin{align*}
  \|f\|_{L^2}^2 \le c_1 \eps^2 + h_{\linop,\Hil}(\hilweight_\nobs)
\end{align*}
where $\{\hilweight_\nobs\}$ are the eigenvalues of the operator
embedding $\Hil$ in $L^2$, $h_{\linop,\Hil}(\cdot)$ is an increasing
function (depending on $\linop$ and $\Hil$) and $c_1 \ge 1$ is some
constant. We considered examples of operators $\linop$ (uniform time
sampling and Fourier truncation) and Hilbert spaces $\Hil$ (Sobolev,
Fourier-type RKHSs) and showed that it is possible to obtain optimal
scaling $h_{\linop,\Hil}(\hilweight_\nobs) =
\mathcal{O}(\hilweight_\nobs)$ in most of those cases. We also
considered random time sampling, under polynomial eigen-decay
$\hilweight_\nobs = \mathcal{O}(n^{-\alpha})$, and effectively showed
that $h_{\linop,\Hil}(\hilweight_\nobs) =
\mathcal{O}(n^{-\alpha/(\alpha+1)})$ (for $\eps$ small enough), with
high probability as $n \to \infty$. This last result complements those
on related quantities obtained by techniques form empirical process
theory, and we conjecture it to be sharp.

\subsection*{Acknowledgements}
AA and MJW were partially supported by NSF Grant CAREER-CCF-0545862 
and AFOSR Grant 09NL184.

\appendix

\section{Analysis of random time sampling}\label{AppRandSamp}

This section is devoted to the proof of
Corollary~\ref{cor:random:samp:1} on random time sampling in
reproducing kernel Hilbert spaces.  The proof is based on an auxiliary
result, which we begin by stating.  Fix some positive integer $\mobs$
and define
   \begin{align}\label{eq:nu:def}
     \nu(\eps) = \nu(\eps;\mobs) := \inf \Big\{ p: \sum_{k >
       \pdim^\mobs} \hilweight_{k} \le \eps^2 \Big\}.
   \end{align}
With this notation, we have
   \begin{lemma}\label{lem:random:samp:1} Assume $\eps^2 < \hilweight_1 
     $ and $32 \,C_\psi^2\,\mobs\, \nu(\eps) \log \nu(\eps) \le
     \nobs$. Then,
     \begin{align}\label{eq:random:samp:1}
       \pr \big\{ \rpp{\SamOp}(\eps) > \Ctpsi\,\eps^2 + \Ctsig\,
       \hilweight_{\nu(\eps)}\big\} \le 2 \exp \Big( -\frac{1}{32
         C_\psi^2}\frac{\nobs}{\nu(\eps)}\Big).
     \end{align}
   \end{lemma}
We prove this claim in Section~\ref{AppLemRandomSamp1} below.

\subsection{Proof of Corollary~\ref{cor:random:samp:1}}

To apply the lemma, recall that we assume that there exists $\mobs$
such that for all (large) $p$, one has
   \begin{align}\label{eq:tail:domination:restate}
     \sum_{k > \pdim^\mobs} \hilweight_k \le \hilweight_{\pdim}.
   \end{align}
   and we let $\mobssig$ be the smallest such $\mobs$. We define
   \begin{align}\label{eq:}
     \muopt(\eps) \defn \inf \big\{ p: \hilweight_{p} \le \eps^2
     \big\},
   \end{align}
   and note that by~(\ref{eq:tail:domination:restate}), we have
   $\nu(\eps;\mobssig) \le \muopt(\eps)$. Then,
   Lemma~\ref{lem:random:samp:1} states that as long as $\eps^2 <
   \hilweight_1$ and $32 \Cpsi^2 \mobssig \muopt(\eps) \log
   \muopt(\eps) \le \nobs$, we have
    \begin{align}\label{eq:random:samp:2}
       \pr \big\{ \rpp{\SamOp}(\eps) > (\Ctpsi  + \Ctsig
       ) \eps^2 \big\} \le 2 \exp \Big(
       -\frac{1}{32 C_\psi^2}\frac{\nobs}{\muopt(\eps)}\Big).
     \end{align}
Now by the definition of $\muopt(\eps)$, we have $\hilweight_j >
\eps^2$ for $j < \muopt(\eps)$, and hence
   \begin{align*}
     \gcal_\nobs^2(\eps) \ge
   \frac{1}{\nobs} \sum_{j \; < \;\muopt(\eps)} \min\{
   \hilweight_j,\eps^2\} =\frac{\muopt(\eps)-1}{\nobs} \,\eps^2 \ge
   \frac{\muopt(\eps)}{2\nobs} \,\eps^2,
   \end{align*}
   since $\muopt(\eps) \ge 2$ when $\eps^2 < \hilweight_1$.  One can
   argue that $\eps \mapsto \gcal_\nobs(\eps)/\eps$ is
   nonincreasing. It follows from definition~(\ref{eq:crit:rad:def})
   that for $\eps \ge \crad_\nobs$, we have
   \begin{align*}
     \muopt(\eps) \le 2\nobs \Big(\frac{\gcal(\eps)}{\eps} \Big)^2 \le
     2\nobs \Big(\frac{\gcal(\crad_\nobs)}{\crad_\nobs} \Big)^2 \le
     2\nobs \crad_\nobs^2,
   \end{align*}
which completes the proof of Corollary~\ref{cor:random:samp:1}.

\subsection{Proof of Lemma~\ref{lem:random:samp:1}}
\label{AppLemRandomSamp1}

For $\xi \in \reals^p$, let $\xi \otimes \xi$ be the rank-one operator
on $\reals^{p}$ given by $\eta \mapsto \iprod{\xi}{ , \eta}_2\,
\xi$. For an operator $A$ on $\reals^p$, let $\mnorm{A}{2}$ denote its
usual operator norm, $\mnorm{A}{2} := \sup_{\|x \|_2 \le 1}\, \|A
x\|_2$. Recall that for a symmetric (i.e., real self-adjoint) operator
$A$ on $\reals^p$, $\mnorm{A}{2} = \sup \{ |\lambda | :
\,\text{$\lambda$ an eigenvalue of $A$} \}$. It follows that
$\mnorm{A}{2} \le \alpha$ is equivalent to $-\alpha I_p \preceq A
\preceq \alpha I_p$.  \newcommand{\PsiSubBlk}[1]{\PsiMat{}^{(#1)}}
\newcommand{\MuSubBlk}[1]{\MuMat{}^{(#1)}}
\newcommand{\eventA}{\mathcal{A}}

Our approach is to first show that $\mnorm{\Psi_p - I_p}{2} \le
\frac{1}{2}$ for some properly chosen $p$ with high probability. It
then follows that $\lammin(\Psi_p) \ge \frac{1}{2}$ and we can use
bound~(\ref{EqnRPWeak}) for that value of $\pdim$. Then, we need to
control $\lamax \big( \MuMat{\pdimt}^{1/2} \MatPsi{\pdimt}^{}
\MuMat{\pdimt}^{1/2} \big)$. To do this, we further partition
$\PsiMat{\pdimt}$ into blocks. In order to have a consistent notation,
we look at the whole matrix $\PsiMat{}$ and let $\PsiSubBlk{k}$
be the principal submatrix indexed by $\{(k-1) \pdim + 1,\dots,(k-1)
\pdim + \pdim\}$, for $k=1,2,\dots,\pdim^{\mobs-1}$. Throughout the
proof, $\mobs$ is assumed to be a fixed positive integer. Also, let
$\PsiSubBlk{\infty}$ be the principal submatrix of $\PsiMat{}$ indexed
by $\{\pdim^\mobs+1,\pdim^\mobs+2,\dots \}$. This provides a full
partitioning of $\PsiMat{}$ for which $\PsiSubBlk{1},\dots,
\PsiSubBlk{\pdim^{\mobs-1}}$ and $\PsiSubBlk{\infty}$ are the diagonal
blocks, the first $\pdim^{\mobs-1}$ of which are $\pdim$-by-$\pdim$
matrices and the last an infinite matrix. To connect with our previous
notations, we note that $\PsiSubBlk{1} = \PsiMat{\pdim}$ and that
$\PsiSubBlk{2}, \dots, \PsiSubBlk{\pdim^{\mobs-1}},
\PsiSubBlk{\infty}$ are diagonal blocks of $\PsiMat{\pdimt}$. Let us
also partition the $\MuMat{}$  matrix and name its diagonal blocks
similarly.

  We will argue that, in fact, we have $\mnorm{\PsiSubBlk{k} - I_p}{2}
  \le \frac{1}{2}$ for all $k=1,\dots,\pdim^{\mobs-1}$, with high
  probability. Let $\eventA_\pdim$ denote the event on which this
  claim holds. In particular, on event $\eventA_\pdim$, we have
  $\PsiSubBlk{k} \preceq \frac{3}{2} I_\pdim$ for $k
  =2,\dots,\pdim^{\mobs-1}$; hence, we can write
  \begin{align}
    \lamax \big( \MuMat{\pdimt}^{1/2} \MatPsi{\pdimt}^{}
  \MuMat{\pdimt}^{1/2} \big) &\le 
  \sum_{k=2}^{\pdim^{\mobs -1}} \lamax 
  \Big( \sqrt{\MuSubBlk{k}} \PsiSubBlk{k} \sqrt{\MuSubBlk{k}} \Big)
  + \lamax 
  \Big( \sqrt{\MuSubBlk{\infty}} \PsiSubBlk{\infty}
  \sqrt{\MuSubBlk{\infty}} \Big) \notag \\
  &\le \frac{3}{2} \sum_{k=2}^{\pdim^{\mobs -1}} \lamax \big(
  \MuSubBlk{k} \big) + \trace \Big(\sqrt{\MuSubBlk{\infty}} \PsiSubBlk{\infty}
  \sqrt{\MuSubBlk{\infty}} \Big) \notag \\
  &= \frac{3}{2} \sum_{k=2}^{\pdim^{\mobs -1}}
  \hilweight_{(k-1)\pdim+1} + \sum_{k \,> \,\pdim^{\mobs}}
  \hilweight_k [\MatPsi{}]_{kk}.
  \label{eq:rand:samp:middle:1}
  \end{align}
  Using assumptions~(\ref{eq:random:samp:assump:sig}) on the sequence
  $\SEQ{\hilweight}$, the first sum can be bounded as
  \begin{align*}
    \sum_{k=2}^{\pdim^{\mobs -1}} \hilweight_{(k-1)\pdim+1} \le
    \sum_{k=2}^{\pdim^{\mobs -1}} \hilweight_{(k-1)\pdim} \le 
    \sum_{k=2}^{\pdim^{\mobs -1}} \Csig
    \hilweight_{k-1}\hilweight_{\pdim} \le \Csig \| \hilweight\|_1 \hilweight_\pdim
  \end{align*}
  Using the uniform boundedness assumption~(\ref{eq:nu:def}), we have
  $[\Psi]_{kk} = \nobs^{-1} \sum_{i=1}^\nobs \psi_k^2(x_i) \le
  \Cpsi^2$. Hence the second sum in~(\ref{eq:rand:samp:middle:1}) is bounded above by $\Cpsi^2 \sum_{k
    > \pdim^{\mobs}} \hilweight_k$.

  We can now apply Theorem~\ref{ThmUpperOne}. Assume for the moment
  that $\eps^2 \ge \sum_{k > \pdim^{\mobs}} \hilweight_k$ so that the
  right-hand side of~(\ref{eq:rand:samp:middle:1}) is bounded above by
  $\frac{3}{2}\Csig \|\hilweight\|_1 \hilweight_\pdim + \Cpsi^2 \eps^2$. Applying
  bound~(\ref{EqnRPWeak}), on event $\eventA_\pdim$, with\footnote{We
    are using the alternate form of the bound based on $(\sqrt{A} +
    \sqrt{B})^2 = \inf_{r \in (0,1)} \big\{ A r^{-1} + B(1-r)^{-1} \big\}$.} $r =
  (1+C_\psi)^{-1}$, we get
\begin{align*}
  \rpp{\SamOp}(\eps^2) &\le 2\Big\{ r^{-1} \eps^2 + (1-r)^{-1} 
  \Big( \frac{3}{2}\Csig \|\hilweight\|_1 \hilweight_\pdim + \Cpsi^2
  \eps^2 \Big)\Big\} + \hilweight_{p+1}  \\
  &= 2(1+C_\psi)^2 \eps^2 + 3(1+\Cpsi^{-1}) \Csig \|\hilweight\|_1
  \hilweight_{p} + \hilweight_{p+1}. \\
  &\le \Ctpsi \,\eps^2 + \Ctsig \,\hilweight_\pdim
\end{align*}
where $\Ctpsi \defn 2(1+C_\psi)^2$ and $\Ctsig \defn 3(1+\Cpsi^{-1})
\Csig \|\hilweight\|_1+1 $. To summarize, we have shown the following
\begin{align}\label{eq:little:implic}
  \text{Event $\eventA_\pdim$} \quad \text{and} \quad \eps^2 \ge
 \sum_{k > \pdim^\mobs} \hilweight_{k} \implies \rpp{\SamOp}(\eps^2) \le
 \Ctpsi\, \eps^2 + \Ctsig\, \hilweight_\pdim.
\end{align}

It remains to control the probability of $\eventA_\pdim \defn \bigcap_{k =
  1}^{\pdim^{\mobs-1}} \big\{ \mnorm{\PsiSubBlk{k} - I_p}{2}
  \le \frac{1}{2} \big\}$. We
start with the deviation bound on $\PsiSubBlk{1} - I_p$, and then extend by
 union bound. We will use the following lemma
which follows, for example, from the Ahlswede-Winter
bound~\cite{Ahlswede2002}, or from~\cite{Rudelson1999}. (See
also~\cite{VershRMT,Tro10,Wigderson2008}.)
\begin{lemma}\label{lem:isotropd:prob:bound}
  Let $\xi_1,\dots,\xi_\nobs$ be i.i.d. random vectors in  $\reals^p$
  with $\ex \xi_1 \otimes \xi_1 = I_p$ and $\| \xi_1 \|_2 \le C_p$
  almost surely for some constant $C_p$. Then, for $\delta \in (0,1)$,
  \begin{align}
    \pr \Big\{ \bigmnorm{\nobs^{-1} \sum_{i=1}^\nobs \xi_i \otimes \xi_i
      - I_p }_{2} > \delta \Big\} \le p \exp \big( -
     \frac{\nobs \delta^2}{4C_p^2} \big).
  \end{align}
\end{lemma}

Recall that for the time sampling operator, $[\linop\,
\psi_k]_i = \frac{1}{\sqrt{\nobs}} \psi_k(x_i)$ so that from~(\ref{eq:Psi:def}),
\begin{align*}
  \Psi_{k \ell} = \frac{1}{\nobs} \sum_{i=1}^\nobs \psi_k(x_i) \psi_\ell(x_i)
\end{align*}
Let $\xi_i :=
(\psi_k(x_i), 1\le k \le p) \in \reals^p$ for $i=1,\dots,\nobs$. Then, $\{\xi_i\}$
satisfy the conditions of Lemma~\ref{lem:isotropd:prob:bound}. In
particular, letting $e_k$ denote the $k$-th standard basis vector of
$\reals^p$, we note that
\begin{align*}
  \iprod{e_k}{\ex (\xi_i \otimes \,\xi_i)e_\ell}_2 =
\ex \iprod{e_k}{\xi_i}_{2} \iprod{e_\ell}{\xi_i}_{2} =
\iprod{\psi_k}{\psi_\ell}_{L^2} = \delta_{k \ell}
\end{align*}
and $\| \xi_{i}
\|_2 \le \sqrt{p} \,C_\psi$, where we have used uniform boundedness of
$\{\psi_k\}$ as in~(\ref{eq:psi:uniform:bounded}). Furthermore, we
have $\PsiSubBlk{1} = n^{-1} \sum_{i=1}^\nobs \xi_i \otimes
\xi_i$. Applying Lemma~\ref{lem:isotropd:prob:bound} with $C_p =
\sqrt{p} C_\psi$ yields,
\begin{align}\label{eq:psi:deviat:temp}
  \pr \big\{ \mnorm{ \PsiSubBlk{1}
      - I_p }{2} > \delta \big\} \; \le\;   \pdim \exp \big(-
     \frac{\delta^2}{4C_\psi^2}\frac{\nobs}{p} \big).
\end{align}
Similar bounds hold for $\PsiSubBlk{k}$,
$k=2,\dots,\pdim^{\mobs-1}$. Applying the union bound, we get
\begin{align*}
  \pr \bigcup_{k=1}^{\pdim^{\mobs-1}} \big\{ \mnorm{ \PsiSubBlk{k}
      - I_p }{2} > \delta \big\} \; \le\;    \exp \big( \mobs \log \pdim -
     \frac{\delta^2}{4C_\psi^2}\frac{\nobs}{p} \big).
\end{align*}

For simplicity, let $A = A_{\nobs,p} := \nobs/(4 C_\psi^2\, p)$. We
impose $\mobs \log p \le \frac{A}{2} \,\delta^2$ so that the exponent
in~(\ref{eq:psi:deviat:temp}) is bounded above by $-\frac{A}{2}
\delta^2$. Furthermore, for our purpose, it is enough to take $\delta
= \frac{1}{2}$. It follows that
\begin{align}\label{eq:prob:bound:temp}
  \pr (\eventA_\pdim^c) = \pr \bigcup_{k=1}^{\pdim^{\mobs-1}} \big\{
  \mnorm{ \PsiSubBlk{k} - I_p }{2} > \frac{1}{2} \big\}
  \; \le\; \exp \big(-
     \frac{1}{32C_\psi^2}\frac{\nobs}{p} \big),
\end{align}
if $32 C_\psi^2 \, \mobs\, p \log p \le \nobs$.
Now, by~(\ref{eq:little:implic}), under $\eps^2 \ge \sum_{k > p^\mobs}\hilweight_{k}$,
$\rpp{\SamOp}(\eps^2) > \Ctpsi\,\eps^2 + \Ctsig\,\hilweight_\pdim$ implies $\eventA_\pdim^c$. Thus, the exponential bound
    in~(\ref{eq:prob:bound:temp}) holds for $\pr
    \{\rpp{\SamOp}(\eps^2) > \Ctpsi \,\eps^2 + \Ctsig\, \hilweight_\pdim\}$ under the assumptions. We are
    to choose $p$ and the bound is optimized by making $p$ as
    small as possible. Hence, we take $p$ to be $\nu(\eps) := \inf
    \{p:\; \eps^2 \ge \sum_{k > \pdim^\mobs}\hilweight_{k}\}$ which proves
    Lemma~\ref{lem:random:samp:1}. (Note that, in general,
    $\nu(\eps)$ takes its values in $\{0,1,2,\dots\}$. The
    assumption $\eps^2 < \hilweight_1$ 
    guarantees that $\nu(\eps) \neq 0$.)


\section{Proof of Lemma~\ref{lem:sp:per}}\label{app:sparse:periodic}
Assume $\hilweight_k = C k^{-\alpha}$, for some $\alpha \ge 2$. First, note the
following upper bound on the tail sum
\begin{align}\label{eq:poly:decay:tail:bound}
  \sum_{k > p} \hilweight_k \le C \int_p^\infty x^{-\alpha} \, dx =
  C_1(\alpha) \,p^{1-\alpha}.
\end{align}
Furthermore, from the bounds~(\ref{eq:sp:per:bnd1})
and~(\ref{eq:sp:per:bnd2}), we have, for $k \ge \nobs + 1$,
\begin{align}\label{eq:sp:per:diag:bound}
  [\PsiMat{}]_{kk} \le \min\{c_1,c_2\}.
\end{align}
To simplify notation, let us define $\periodidx \defn \{1, 2,\dots,
\gamma \nobs\}$.

Consider the case $\alpha > 2$. We will use the
$\ell_\infty$--\,$\ell_\infty$ upper bound of~(\ref{eq:infnorm:bound}),
with $\pdim = \nobs$. Fix
some $k \ge \nobs + 1$. Note that $\hilweight_k \le \hilweight_{\nobs
  + 1}$. Then, recalling the assumptions on $\PsiMat{}$ and 
the definition of $S_k$, we have
\begin{align}
  \sum_{ \ell \ge \nobs + 1}
  \sqrt{\hilweight_{k}}\sqrt{\hilweight_{\ell}}\, \big|
  [\PsiMat{}]_{k,\ell} \big| &\le \sqrt{\hilweight_{\nobs+1}}
  \sum_{q=0}^\infty \sum_{r = 1}^{\gamma \nobs}
  \sqrt{\hilweight_{\nobs+ r+q \gamma \nobs}}
  \big| [\PsiMat{}]_{k,\nobs+r + q \gamma \nobs} \big| \notag \\
  &= \sqrt{\hilweight_{\nobs+1}} \sum_{q=0}^\infty \sum_{r =
    1}^{\gamma \nobs} \sqrt{\hilweight_{\nobs+ r+q \gamma \nobs}}
  \big| [\PsiMat{}]_{k,\nobs+r} \big| \notag \\
  &\le \sqrt{\hilweight_{\nobs+1}} \sum_{q=0}^\infty \Big\{ c_1
  \sum_{r \,\in\, S_k} \sqrt{\hilweight_{\nobs+ r+q \gamma \nobs}} +
  \frac{c_2}{\nobs} \sum_{r \,\in\, \periodidx \setminus S_k}
  \sqrt{\hilweight_{\nobs+ r+q \gamma \nobs}} \Big\}.\label{eq:sp:per:mid:a1}
\end{align}
Using~(\ref{eq:poly:decay:tail:bound}), the second double sum
in~(\ref{eq:sp:per:mid:a1}) is bounded by
\begin{align}
   \sum_{q=0}^\infty \;\sum_{r \,\in\, \periodidx \setminus S_k}
  \sqrt{\hilweight_{\nobs+ r+q \gamma \nobs}} \; \le \;
  \sum_{\ell > \nobs} \sqrt{\hilweight_\ell} \; \le\; C_2(\alpha)
  \,\nobs^{1-\alpha/2}.
  \label{eq:sp:per:mid:a2}
\end{align}
Recalling that $S_k \subset \periodidx$ and $|S_k|= \eta$, the first double sum in~(\ref{eq:sp:per:mid:a1}) can be
bounded as follows
\begin{align}
  \sum_{q=0}^\infty \;\;\sum_{r \,\in\, S_k} \sqrt{\hilweight_{\nobs+
      r+q \gamma \nobs}}
  &= \sqrt{C} \sum_{q=0}^\infty \;\;\sum_{r \,\in\, S_k} {(\nobs+
      r+q \gamma \nobs)}^{-\alpha/2} \notag\\
  &\le \sqrt{C} \sum_{q=0}^\infty \;\;\sum_{r \,\in\, S_k} {(\nobs
     +q \gamma \nobs)}^{-\alpha/2} \notag \\
  &\le \sqrt{C} \,\eta \,\sum_{q=0}^\infty  {(1+q \gamma
   )}^{-\alpha/2}  \nobs^{-\alpha/2} \notag \\
  &\le \sqrt{C} \, \eta \Big( 1 + \,\gamma^{-\alpha/2} 
  \sum_{q=1}^\infty q^{-\alpha/2}
  \Big) \nobs^{-\alpha/2} \notag \\
  &= C_3(\alpha,\gamma,\eta) \,\nobs^{-\alpha/2}\label{eq:sp:per:mid:a3}
\end{align}
where in the last line we have used $\sum_{q=1}^\infty q^{-\alpha/2} <
\infty$ due to $\alpha/2 >
1$. Combining~(\ref{eq:sp:per:mid:a1}),~(\ref{eq:sp:per:mid:a2})
and~(\ref{eq:sp:per:mid:a3}) and noting that
$\sqrt{\hilweight_{\nobs+1}} \le \sqrt{C} \nobs^{-\alpha/2}$, we obtain
\begin{align}
  \sum_{ \ell \ge \nobs + 1}
  \sqrt{\hilweight_{k}}\sqrt{\hilweight_{\ell}} \,\big|
  [\PsiMat{}]_{k,\ell} \big| &\le \sqrt{C} \nobs^{-\alpha/2} \Big\{
 c_1 C_3(\alpha,\gamma,\eta)\,  \nobs^{-\alpha/2} + \frac{c_2}{\nobs}\,
 C_2(\alpha)\, \nobs^{1-\alpha/2}\Big\} = C_4(\alpha,\eta,\gamma)\, \nobs^{-\alpha}.
\end{align}
Taking supremum over $k \ge 1$ and applying the $\ell_\infty$--\,$\ell_\infty$ bound of~(\ref{eq:infnorm:bound}),
with $\pdim = \nobs$, concludes the proof of part~(a).

Now, consider the case $\alpha=2$. The above argument breaks down in
this case because $\sum^{\infty}_{q=1} q^{-\alpha/2}$ does not
converge for $\alpha =2$. A remedy is to further partition the matrix
$ \MuMat{\nobst}^{1/2} \MatPsi{\nobst}^{} \MuMat{\nobst}^{1/2} $
. Recall that the rows and columns of this matrix are indexed by
$\{\nobs+1,\nobs+2,\dots \}$. Let $A$ be the principal submatrix
indexed by $\{\nobs+1,\nobs+2,\dots,\nobs^2\}$ and $D$ be the
principal submatrix indexed by $\{\nobs^2+1, \nobs^2
+2, \dots \}$. We will use a combination of the
bounds~(\ref{eq:sp:per:bnd1}) and~(\ref{eq:sp:per:bnd2}), and the
well-known perturbation bound $\lamax \big[ \big(
\begin{smallmatrix}
  A & C \\ C^T & D
\end{smallmatrix}
\big)\big] \le \lamax(A) + \lamax(D)$, to write
\begin{align}\label{eq:further:part}
  \lamax \big( \MuMat{\nobst}^{1/2}
    \MatPsi{\nobst}^{} \MuMat{\nobst}^{1/2} \big) \le \lamax(A) +
    \lamax(D)
    \le \mnorm{A}{\infty} + \trace(D).
\end{align}
The second term is bounded as
\begin{align}\label{eq:trace:D:bound}
  \trace(D) = \sum_{ k > \nobs^2} \hilweight_k \, 
  [\PsiMat{}]_{kk} \le \min\{c_1,c_2\} \sum_{ k > \nobs^2}
  \hilweight_k =  \min\{c_1,c_2\} \, (\nobs^2)^{1-2} =
  C_5(\gamma) \, \nobs^{-2},
\end{align}
where we have used~(\ref{eq:poly:decay:tail:bound})
and~(\ref{eq:sp:per:diag:bound}). To bound the first term, fix $k \in
\{ \nobs + 1,\dots, \nobs^2\}$. By an argument
similar to that of part~(a) and noting that $\gamma \ge 1$, hence
$\gamma \nobs^2 \ge \nobs^2 $, we have
\begin{align}
  \sum_{ \ell = \nobs + 1}^{\nobs^2}
  \sqrt{\hilweight_{k}}\sqrt{\hilweight_{\ell}}\, \big|
  [\PsiMat{}]_{k,\ell} \big| 
  &\le \sqrt{\hilweight_{\nobs+1}}
  \sum_{q=0}^{\nobs} \sum_{r = 1}^{\gamma \nobs}
  \sqrt{\hilweight_{\nobs+ r+q \gamma \nobs}}
  \big| [\PsiMat{}]_{k,\nobs+r} \big| \notag \\
  &\le \sqrt{\hilweight_{\nobs+1}} \sum_{q=0}^{\nobs} \Big\{ c_1
  \sum_{r \,\in\, S_k} \sqrt{\hilweight_{\nobs+ r+q \gamma \nobs}} +
  \frac{c_2}{\nobs} \sum_{r \,\in\, \periodidx \setminus S_k}
  \sqrt{\hilweight_{\nobs+ r+q \gamma \nobs}} \Big\}.\label{eq:sp:per:mid:b1}
\end{align}
Using $\gamma
\ge 1$ again, the second double sum in~(\ref{eq:sp:per:mid:b1}) is bounded as
\begin{align}
  \sum_{q=0}^{\nobs} \;\;\sum_{r \,\in\, \periodidx \setminus S_k}
  \sqrt{\hilweight_{\nobs+ r+q \gamma \nobs}} \le \sum_{\ell =
    \nobs+1}^{3 \gamma \nobs^2}
  \sqrt{\hilweight_{\ell}} \le 
  \; \sqrt{C} \sum_{\ell = 2}^{3 \gamma \nobs^2}
  \frac{1}{\ell} \le \sqrt{C} \log (3\gamma \nobs^2) \le 
  C_6(\gamma) \log \nobs,\label{eq:sp:per:mid:b2}
\end{align}
for sufficiently large $n$. Note that we have used the bound
$\sum_{\ell=2}^p \ell^{-1} \le \int_{1}^p x^{-1}\,dx = \log p$.
The first double sum in~(\ref{eq:sp:per:mid:b1}) is bounded as follows
\begin{align}
  \sum_{q=0}^\infty \;\;\sum_{r \,\in\, S_k} \sqrt{\hilweight_{\nobs+
      r+q \gamma \nobs}}
  &= \sqrt{C} \sum_{q=0}^{\nobs} \;\;\sum_{r \,\in\, S_k} {(\nobs+
      r+q \gamma \nobs)}^{-1} \notag\\
  &\le \sqrt{C} \,\eta \,\sum_{q=0}^{\nobs}  {(1+q \gamma
   )}^{-1}  \nobs^{-1} \notag \\
  &\le \sqrt{C} \, \eta \Big( 1 + \,\gamma^{-1} 
  + \gamma^{-1} \sum_{q=2}^{\nobs} q^{-1}
  \Big) \nobs^{-1} \notag \\
  &= C_7(\gamma,\eta) \,\nobs^{-1} \log \nobs,\label{eq:sp:per:mid:b3}
\end{align}
for $\nobs$ sufficiently
large. Combining~(\ref{eq:sp:per:mid:b1}),~(\ref{eq:sp:per:mid:b2})
and~(\ref{eq:sp:per:mid:b3}), taking supremum over $k$ and using the simple
bound $\sqrt{\hilweight_{\nobs+1}} \le \sqrt{C} \nobs^{-1}$, we get
\begin{align}
  \mnorm{A}{\infty} \le \sqrt{C} \nobs^{-1} \Big\{
 c_1 C_7(\gamma,\eta)\,  \frac{\log n}{\nobs} + \frac{c_2}{\nobs}\,
 C_6(\gamma)\, \log \nobs\Big\} = C_8(\gamma,\eta)\, \frac{\log
   \nobs}{\nobs^2}
\end{align}
which in view of~(\ref{eq:trace:D:bound}) and~(\ref{eq:further:part}) completes the proof of part~(b).


\section{Relationship between $\rp(\eps)$ and $\tpl(\eps)$}
\label{app:lower:quant}

In this appendix, we prove the claim made in Section~\ref{sec:intro}
about the relation between the upper quantities $\rp$ and $\tp$ and
the lower quantities $\tpl$ and $\rpl$. We only carry out the proof
for $\rp$; the dual version holds for $\tp$. To simplify the argument,
we look at slightly different versions of $\rp$ and $\tpl$, defined as
\begin{align}
\rpo (\eps) \;&:=\; \sup \big\{ \|f\|_\Lt^2: \; f \in \ballh, \;
\|f\|_\linop^2\; < \eps^2 \big\}, \\ \tplo(\delta) \;&:=\; \inf
\;\big\{ \|f\|_\linop^2\;: \; f \in \ballh, \; \|f\|_\Lt^2 >
\delta^2\big\}
\end{align}
and prove the following
\begin{align}\label{eq:rpo:tplo}
  {\rpo}^{-1}(\delta) = \tplo(\delta)
\end{align}
where ${\rpo}^{-1}(\delta) := \inf \{ \eps^2 :\; \rpo(\eps) >
\delta^2\}$ is a generalized inverse of $\rpo$. To
see~(\ref{eq:rpo:tplo}), we note that $\rp(\eps) > \delta^2$ iff
there exists $f \in \ball_\Hil$ such that $\|f\|^2_\linop < \eps^2$
and $\|f\|^2_{L^2} > \delta^2$. But this last statement is equivalent
to $\tplo(\delta) < \eps^2$. Hence,
\begin{align}
   {\rpo}^{-1}(\delta) = \inf \{ \eps^2 :\; \tplo(\delta) < \eps^2
   \}
\end{align}
which proves~(\ref{eq:rpo:tplo}).

Using the following lemma, we can use relation~(\ref{eq:rpo:tplo}) to
convert upper bounds on $\rp$ to lower bounds on $\tpl$.
\begin{lemma}\label{lem:gen:inv}
  Let $t \mapsto p(t)$ be a nondecreasing function (defined on the
  real line with values in the extended real line.). Let $q$ be its
  generalized inverse defined as $q(s) := \inf \{ t:\; p(t) >
  s\}$. Let $r$ be a properly invertible (i.e., one-to-one) function
  such that $p(t) \le r(t)$, for all $t$. Then,
  \begin{itemize}
  \item[(a)] $q(p(t)) \ge t$, for all $t$,
  \item[(b)] $q(s) \ge r^{-1}(s)$, for all $s$.
  \end{itemize}
\end{lemma}
\begin{proof}
   Assume~(a) does not hold, that is, $\inf\{ \alpha:\, p(\alpha) > p(t)\}
   < t$. Then, there exists $\alpha_0$ such that $p(\alpha_0) > p(t)$
   and $\alpha_0 < t$. But this contradicts $p(t)$ being
   nondecreasing. For part~(b), note that~(a) implies $t \le q(p(t))
   \le q(r(t))$, since $q$ is nondecreasing by definition. Letting $t
   := r^{-1}(s)$ and noting that $r(r^{-1}(s)) = s$, by assumption,
   proves~(b).
\end{proof}

Let $p = \rpo$, $q = \tplo$ and $r(t) = A t+
B$ for some constant $A > 0$. Noting that $\rpo \le \rp$ and $\tpl(\cdot + \gamma) \ge
\tplo$ for any $\gamma > 0$, we obtain
from Lemma~\ref{lem:gen:inv} and~(\ref{eq:rpo:tplo}) that
\begin{align}\label{eq:rp:tpl:bound:translation}
  \rp(\eps) \le A \,\eps^2 + B \; \implies \;
  \tpl(\delta+) \ge  \frac{\delta^2}{A} - B,
\end{align}
where $\tpl(\delta+)$ denotes the right limit of $\tpl$ as
$\delta^2$.  
This may be used to 
translate an upper bound of the form~(\ref{EqnRPWeak}) on $\rp$ to a
corresponding lower bound on $\tpl$. 


\section{The $2\times 2$ subproblem\label{app:2by2}}
The following subproblem arises in the proof of
Theorem~\ref{ThmUpperOne}.
\begin{align}\label{eq:2by2:subp}
  F(\eps^2) := \sup \Big\{
  \underbrace{\begin{pmatrix}
    r & s
  \end{pmatrix}
  \begin{pmatrix}
    u^2 & 0 \\
    0 & v^2
  \end{pmatrix}
  \begin{pmatrix}
    r \\ s
  \end{pmatrix}}_{=: \;x(r,s)} :\;
  r^2 + s^2 \le 1, \; 
  \underbrace{\begin{pmatrix}
    r & s
  \end{pmatrix}
  \begin{pmatrix}
    a^2 & 0 \\
    0 & d^2
  \end{pmatrix}
  \begin{pmatrix}
    r \\ s
  \end{pmatrix}}_{=:\; y(r,s)} \le \eps^2
\Big\},
\end{align}
where $u^2,v^2,a^2$ and $d^2$ are given constants and the optimization
is over $(r,s)$.
Here, we discuss the solution in some detail; in particular, we provide
explicit formulas for $F(\eps^2)$. Without loss of generality assume
$u^2 \ge v^2$. Then, it is clear that $F(\eps^2) \le u^2$ and
$F(\eps^2) = u^2$ for $\eps^2 \ge u^2$. Thus, we are interested in what
happens when $\eps^2 < u^2$.

The problem is easily solved by drawing a picture. Let $x(r,s)$ and
$y(r,s)$ be as denoted in the last display. Consider the set
\begin{align}
  \scal &:= \big\{ \big(x(r,s), \,y(r,s) \big):\; r^2 + s^2 \le 1\} \notag
  \\ &\;=
  \big\{ r^2 (u^2,a^2) + s^2 (v^2,d^2) + q^2(0,0):\; r^2 + s^2 + q^2 =
  1 \big\}
  \notag \\
  &\;= \conv \big\{(u^2,a^2),\, (v^2,d^2),\, (0,0)\big\}.\label{eq:scal:def}
\end{align}
That is, $\scal$ is the convex hull of the three points $(u^2,a^2)$,
$(v^2,d^2)$ and the origin $(0,0)$.
 
Then, two (or maybe
three) different
pictures arise depending on whether $a^2 > d^2$ (and whether $d^2 \ge
v^2$ or $d^2 < v^2$) or $a^2 \le d^2$; see Fig.~\ref{fig:1}.
\begin{figure}[tb]
  \centering
  \includegraphics[scale=0.4]{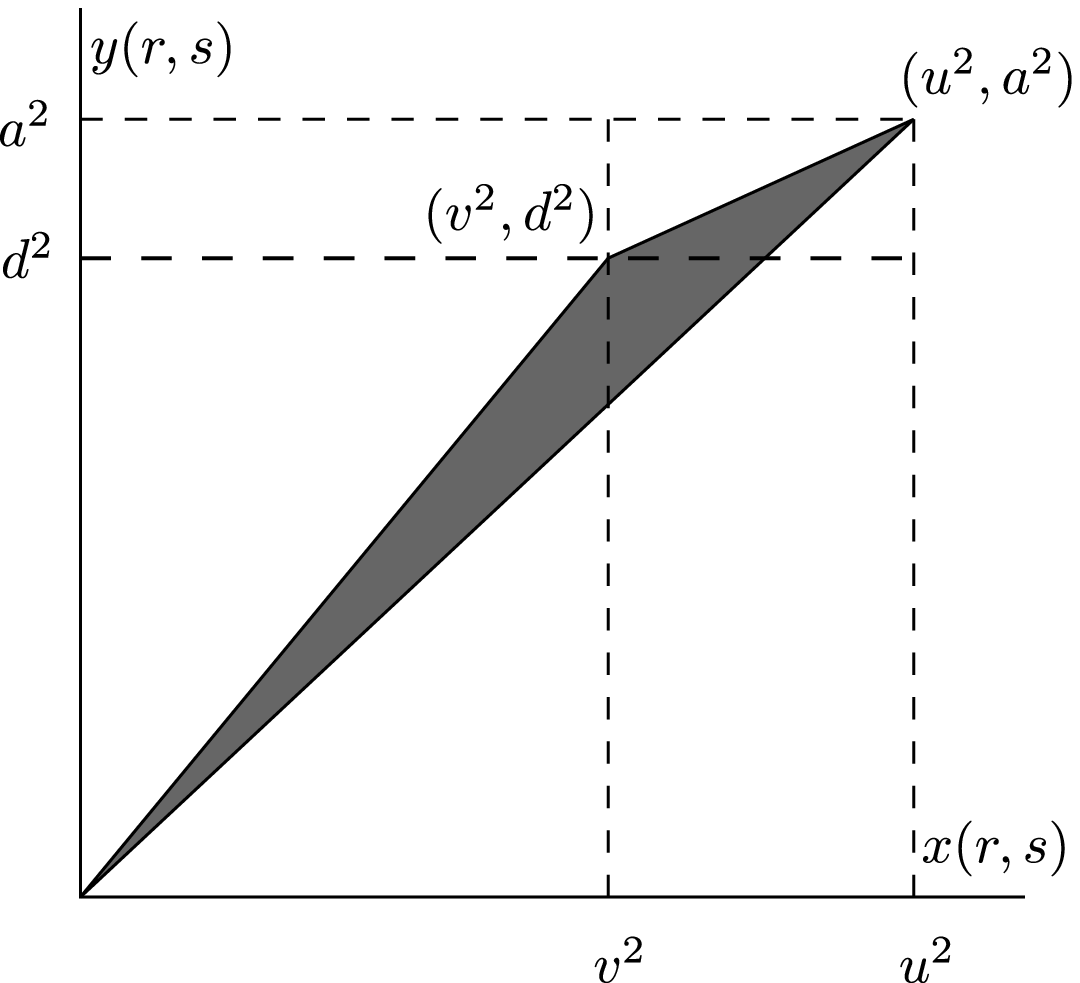}
  \includegraphics[scale=0.4]{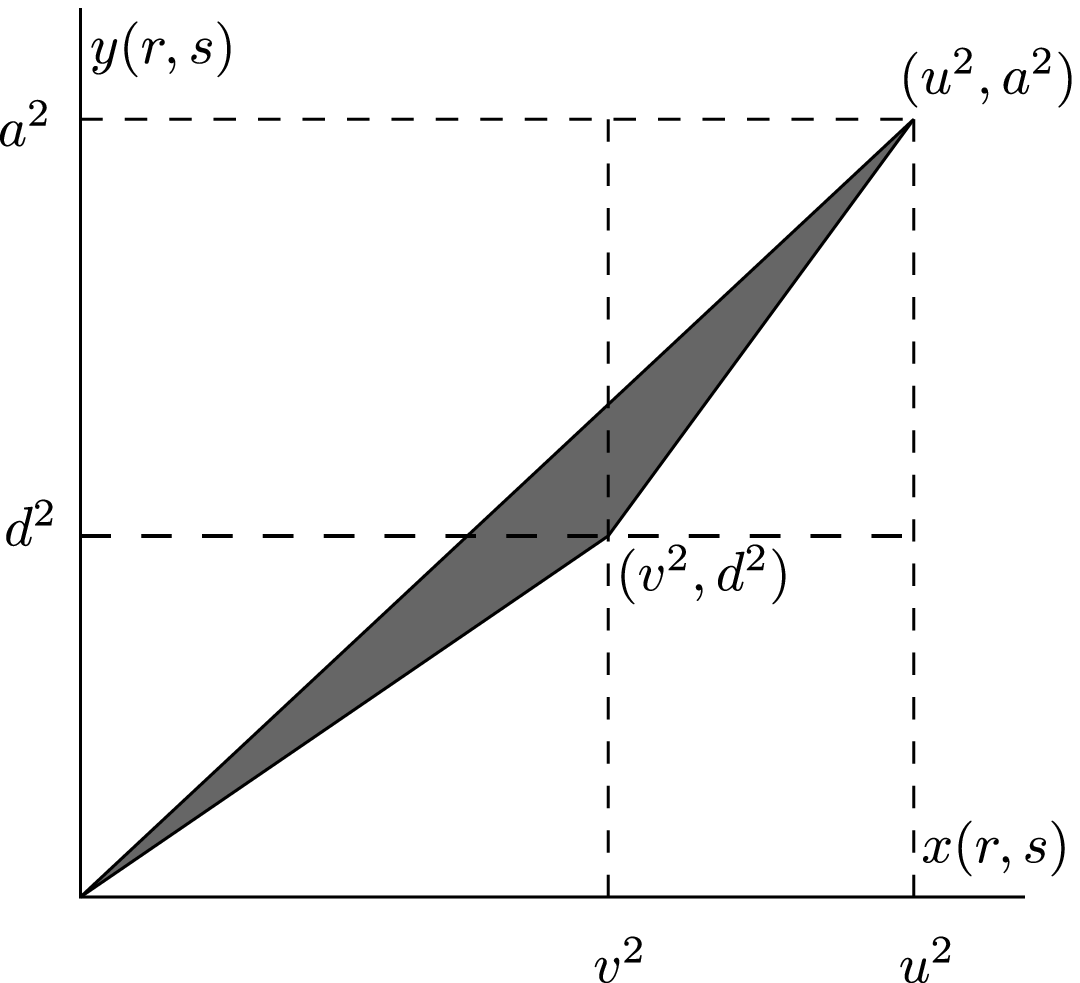}
  \includegraphics[scale=0.4]{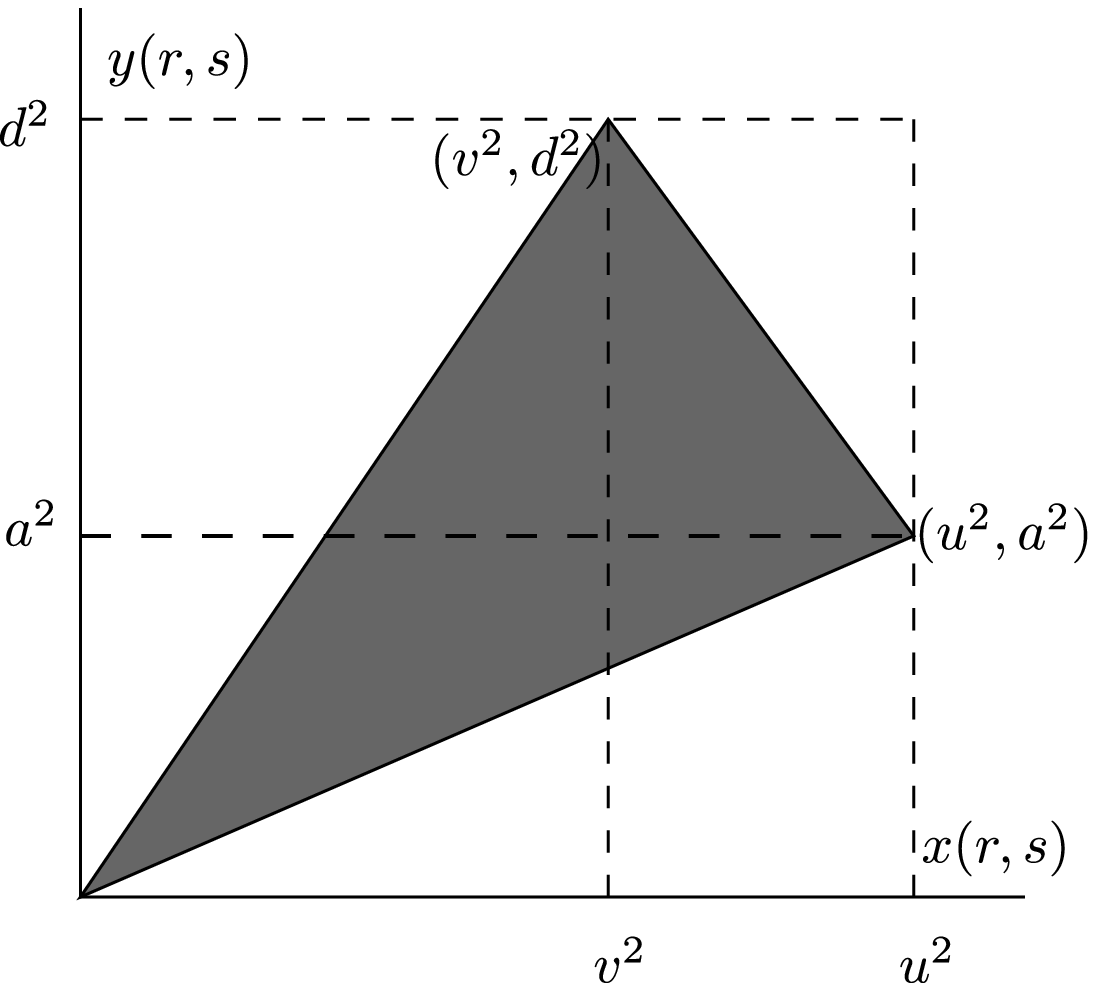} \\
  \includegraphics[scale=0.44]{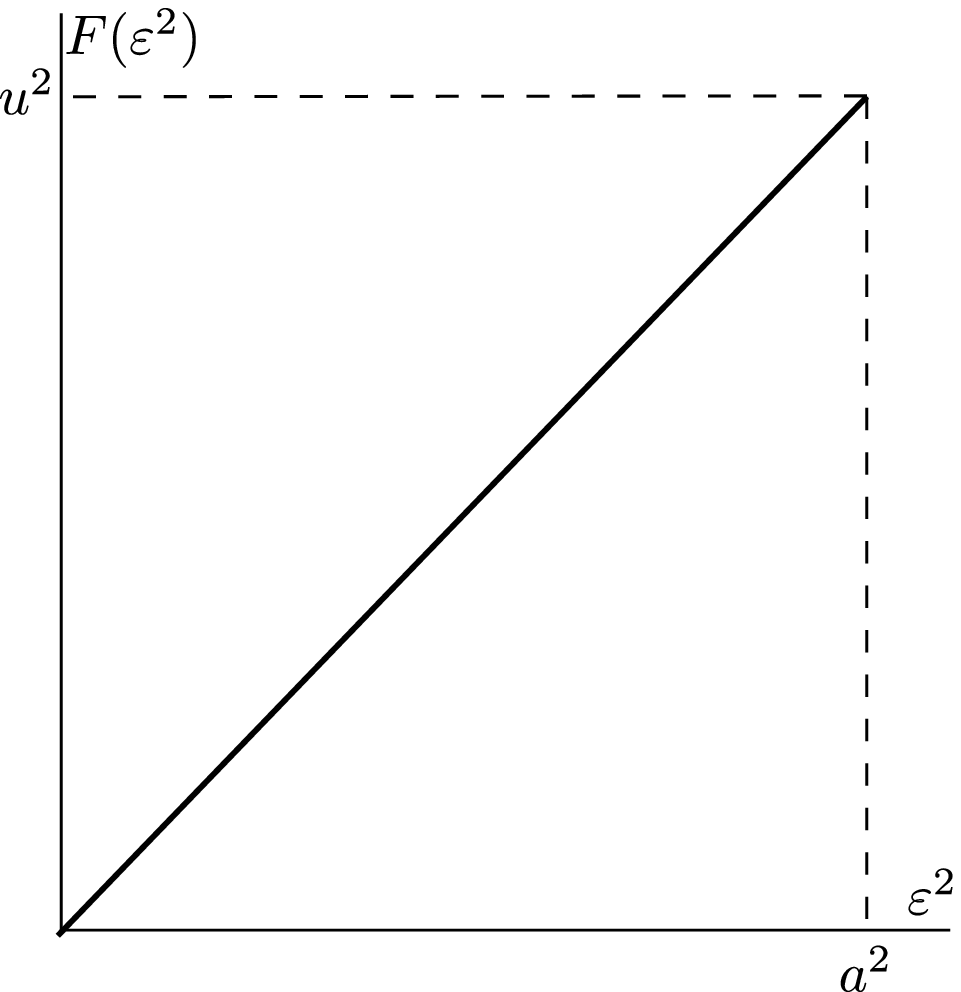}
  \includegraphics[scale=0.44]{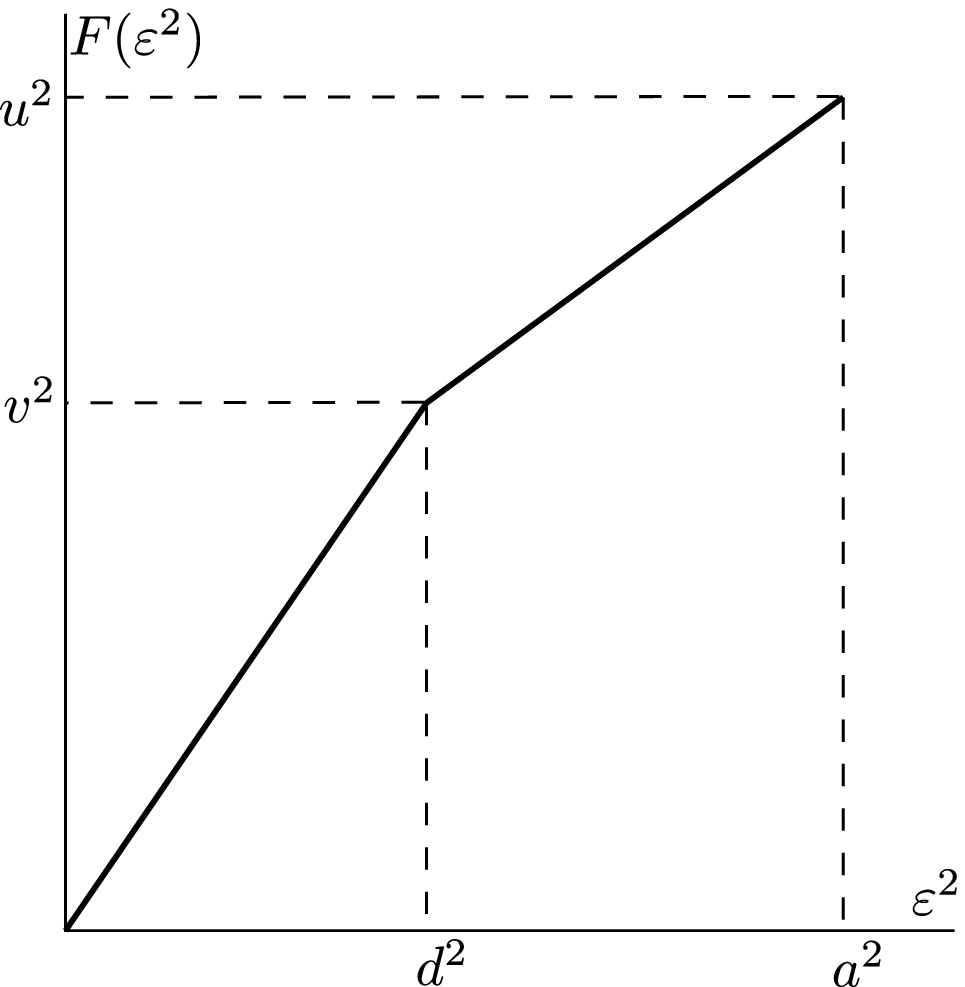}
  \includegraphics[scale=0.44]{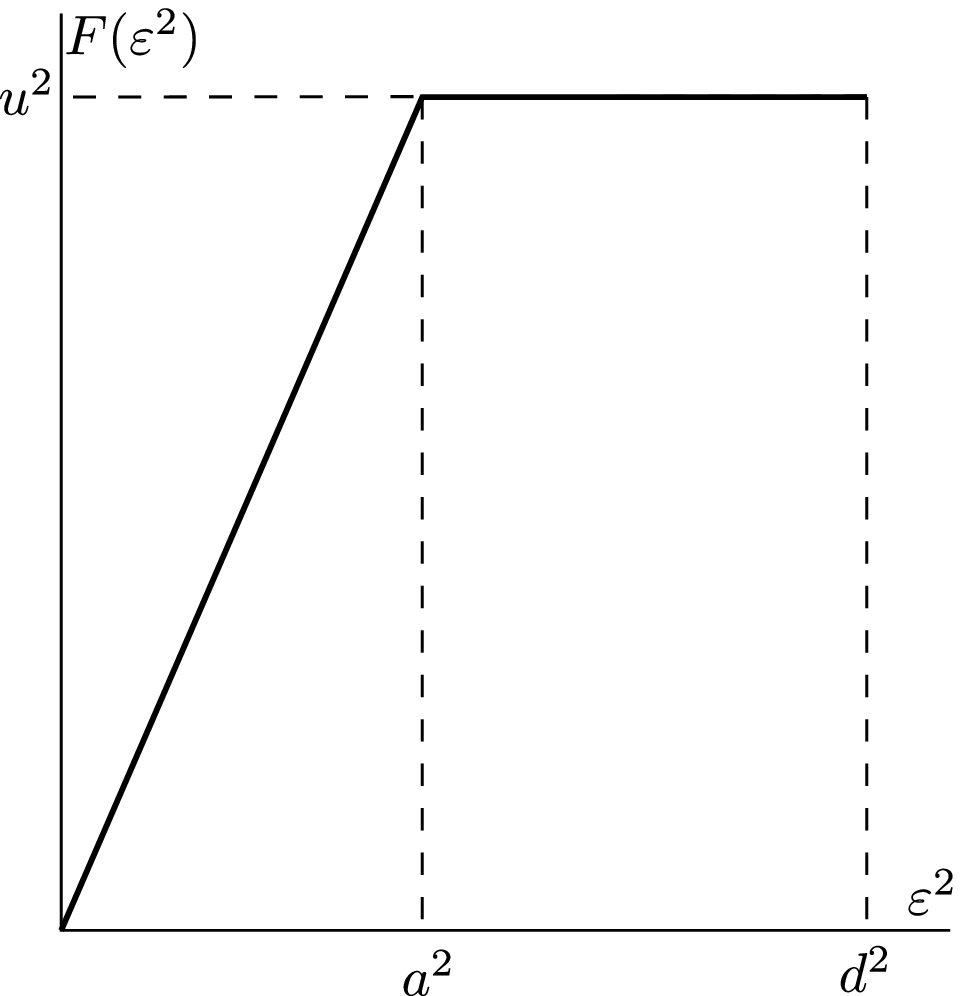}
  \caption{Top plots illustrate the set $\scal$ as defined in~(\ref{eq:scal:def}), in various
    cases. The bottom plots are the corresponding $\eps^2 \mapsto F(\eps^2)$.}
  \label{fig:1}
\end{figure}
It follows that we have two (or three) different pictures for the function
$\eps^2 \mapsto F(\eps^2)$. In particular, for $a^2 > d^2$
and $d^2 < v^2$,
\begin{align}\label{eq:F:eps2:a2:gr:d2}
  F(\eps^2) = v^2 \min \Big\{ \frac{\eps^2}{d^2},1\Big\}
  + (u^2 -v^2) \max \Big\{ 0, \frac{\eps^2-d^2}{a^2 - d^2}\Big\},
\end{align}
for $a^2 > d^2$ and $d^2 \ge v^2$, $F(\eps^2) = \eps^2$, and for $a^2
\le d^2$,
\begin{align*}
  F(\eps^2) = u^2 \min \Big\{ \frac{\eps^2}{a^2}, 1\Big\}.
\end{align*}
All the equations above are valid for $\eps^2 \in [0, \hilweight_1]$.


\section{Details of the Fourier truncation example\label{app:Four:details}}
Here we establish the claim that the bound~(\ref{eq:four:rp:bound})
holds with equality. Recall that for the (generalized) Fourier
truncation operator $\Four$, we have
\begin{align*}
  \rpp{\Four}(\eps^2) = \sup \Big\{ \sum_{k=1}^\infty \hilweight_k
  \alpha_k^2:\; \sum_{k=1}^\infty \alpha_k^2 \le 1, \; \sum_{k=1}^n
  \hilweight_k \alpha_k^2 \le \eps^2 \Big\}
\end{align*}
Let $\alpha = (t \xi, s \gamma)$, where $t,s \in \reals$, $\xi =
(\xi_1, \dots,\xi_n) \in \reals^n$, $\gamma =
(\gamma_1,\gamma_2\dots) \in \elltwo$ and $\|\xi\|_2 = 1 =
\|\gamma\|_2$. Let $u^2 = u^2(\xi) := \sum_{k=1}^n \hilweight_k \xi_k^2$ and
$v^2 = v^2(\gamma) := \sum_{k> n} \hilweight_k \gamma_k^2$. 

Let us fix $\xi$ and $\gamma$ for now and try to optimize over $t$ and $s$.
That is, we look at
\begin{align*}
  G(\eps^2; \xi, \gamma) :=\sup \Big\{ t^2 u^2 + s^2 v^2 :\; t^2 + s^2 \le 1, \; t^2 u^2 \le \eps^2\Big\}.
\end{align*}
This is an instance of the $2$-by-$2$ problem~(\ref{eq:2by2:subp}), with $a^2
= u^2$ and $d^2 = 0$. Note that our assumption that $u^2 \ge v^2$
holds in this case, for all $\xi$ and $\gamma$, because $\SEQ{\hilweight}$ is a nonincreasing sequence.
Hence, we have, for $\eps^2 \le \hilweight_1$,
\begin{align*}
  G(\eps^2; \xi, \gamma) = v^2 + (u^2 - v^2) \frac{\eps^2}{u^2} = v^2(\gamma)+
 \Big( 1- \frac{v^2(\gamma)}{u^2(\xi)} \Big) \eps^2.
\end{align*}

Now we can maximize $G(\eps^2; \xi, \gamma)$ over $\xi$ and then
$\gamma$. Note that $G$ is increasing in $u^2$. Thus, the maximum is
achieved by selecting $u^2$ to be $\sup_{\|\xi\|_2 =1} u^2(\xi) =
\hilweight_1$. Thus,
\begin{align*}
   \sup_{\xi} G(\eps^2; \xi, \gamma) = \Big(1 - \frac{\eps^2}{\hilweight_1}\Big) v^2
   (\gamma) + \eps^2.
\end{align*}
For $\eps^2 < \hilweight_1$, the above is increasing in $v^2$. Hence the
maximum is achieved by setting $v^2$ to be $\sup_{\|\gamma\|_2 = 1}
v^2(\gamma) = \hilweight_{n+1}$. Hence, for $\eps^2 \le \hilweight_1$
\begin{align}\label{eq:Four:exact:appendix}
  \rpp{\Four}(\eps^2) := \sup_{\xi, \,\gamma}G(\eps^2; \xi, \gamma)
   = \Big(1 - \frac{\hilweight_{n+1}}{\hilweight_1}\Big) \eps^2 + \hilweight_{n+1}.
\end{align}

\section{An quadratic inequality\label{app:quad:ineq}}
In this appendix, we derive an inequality
which will be used in the proof of Theorem~\ref{ThmUpperOne}. 
Consider a positive semidefinite matrix $M$ (possibly
infinite-dimensional) partitioned as 
\begin{align*}
  M = 
  \begin{pmatrix}
    A & C \\ C^T & D
  \end{pmatrix}.
\end{align*}
Assume that there exists $\rho^2 \in (0,1)$ and $\kappa^2 > 0$ such
that
\begin{align}\label{eq:quad:ineq:cond:1}
    \begin{pmatrix}
      A & C \\
      C^T & (1-\rho^2) D + \kappa^2 I \\
    \end{pmatrix} \succeq 0.
\end{align}
Let $(x,y)$ be a vector partitioned to match the block structure of $M$. Then we have the following. 
\begin{lemma}
  Under~(\ref{eq:quad:ineq:cond:1}), for all $x$ and $y$,
  \begin{align}\label{eq:quad:ineq}
    x^T A x + 2x^T C y +  y^T D y  \;\ge\; \rho^2 x^T A x -
    \frac{\kappa^2}{1- \rho^2}
    \|y\|_2^2.
  \end{align}
\end{lemma}
\begin{proof}
  By assumption~(\ref{eq:quad:ineq:cond:1}), we have
  \begin{align}
    \begin{pmatrix}
      \sqrt{1-\rho^2} \, x^T & \frac{1}{\sqrt{1-\rho^2}} \, y^T
    \end{pmatrix}
    \begin{pmatrix}
      A & C \\
      C^T & (1-\rho^2) D + \kappa^2 I \\
    \end{pmatrix} 
    \begin{pmatrix}
      \sqrt{1-\rho^2}\, x \\ \frac{1}{\sqrt{1-\rho^2}}\, y
    \end{pmatrix} \; \ge \;0.
  \end{align}
\end{proof}
Writing~(\ref{eq:quad:ineq:cond:1}) as a perturbation of the original
matrix,
  \begin{align}
    \begin{pmatrix}
      A & C \\
      C^T & D \\
    \end{pmatrix} + \begin{pmatrix}
      0 & 0 \\
      0 & -\rho^2 D + \kappa^2 I \\
    \end{pmatrix}\succeq 0,
  \end{align}
we observe that a sufficient condition for~(\ref{eq:quad:ineq:cond:1})
to hold is $\rho^2 D \preceq \kappa^2 I$. That is, it is sufficient to
have
\begin{align}\label{eq:quad:ineq:cond:2}
    \rho^2 \lamax(D) \le \kappa^2.
  \end{align}
Rewriting~(\ref{eq:quad:ineq:cond:1}) differently, as
\begin{align}
  \begin{pmatrix}
      (1-\rho^2) A & 0 \\
      0 & (1-\rho^2) D \\
    \end{pmatrix} + \begin{pmatrix}
      \rho^2 A & C \\
      C^T & \kappa^2 I \\
    \end{pmatrix}\succeq 0,
\end{align}
we find another sufficient condition for~(\ref{eq:quad:ineq:cond:1}),
namely, $\rho^2 A - \kappa^{-2} C C^T \succeq 0$. In particular, it is
also sufficient to have
\begin{align}
  \kappa^{-2} \lamax(CC^T) \le \rho^2 \lamin(A).
\end{align}



\bibliographystyle{elsarticle-num}
\bibliography{approx_paper_mod}







\end{document}